\documentclass[12pt]{amsart}
\title[Quasisymmetries of the basilica]{Quasisymmetries of the basilica\\ 
and the Thompson Group}
\author{Mikhail Lyubich and Sergei Merenkov}

\address{Mikhail Lyubich\\Institute for Mathematical Sciences\\
Stony Brook University\\
Stony Brook, NY 11794\\USA }
\email{mlyubich@math.stonybrook.edu}
\thanks{M.L. was supported by NSF grants 
DMS-1007266, DMS-1301602, DMS-1600519.}

\address{Sergei Merenkov\\
Department of Mathematics\\City College of New York and CUNY Graduate Center,  New York, NY 10031, USA
}
\email{%merenkov@illinois.edu, 
smerenkov@ccny.cuny.edu}
\thanks{S.M. is grateful for the support by the Simons Foundation at the IMS}

\thanks{}

\usepackage[matrix,arrow,curve,cmtip]{xy}
\usepackage{amsmath, amssymb, amsthm,latexsym}
\usepackage{graphicx}
\usepackage{caption}
\usepackage{subcaption}
\usepackage{ftnxtra}

\newcommand\C{{\mathbb C}}
\newcommand\hC{\widehat {\mathbb C}}
\newcommand\N{{\mathbb N}}

\newcommand\D{{\mathbb D}}

\newcommand\Z{{\mathbb Z}}
\newcommand\R{{\mathbb R}}

\newcommand\T{{\mathbb T}}

\newcommand\Ju{\mathcal J}

\newcommand\dee{\partial}

\newcommand\id{\operatorname{id}}

\renewcommand\:{\colon}

\newcommand\ra {\rightarrow}

\newcommand{\KK}{{\mathcal{K}}}
\newcommand{\bga}{\boldsymbol{\gamma}}
\newcommand{\di}{\partial}
\newcommand{\inter}{\operatorname{int}}

\newcommand\eps{\epsilon}

\newcommand\no{\noindent}

\numberwithin{equation}{section}

\newtheorem{theorem}{Theorem}[section]

\newtheorem{corollary}[theorem]{Corollary}

\newtheorem{lemma}[theorem]{Lemma}

\theoremstyle{definition}

\newcommand{\comm}[1]{}

\newcommand{\FF}{{\mathcal{F}}}
\newcommand{\ZZ}{{\mathcal{Z}}}
\newcommand{\YY}{{\mathcal{Y}}}
\newcommand{\sm}{\setminus} 
\newcommand{\de}{\delta}
\newcommand{\cl}{\operatorname{cl}}
\newcommand{\tl} {\tilde}
\newcommand{\Ups}{\Upsilon} 
\newcommand{\msk}{\medskip} 
\newcommand{\bal}{{\boldsymbol{\alpha}}}
\newcommand{\inv}{\psi}

\newcommand{\Dyn}{\operatorname{Dyn}}
\newcommand{\TT}{{\mathfrak{T}}}

\begin{document}

\abstract{
We give a description of the group of all quasisymmetric self-maps of the Julia set of $f(z)=z^2-1$ that have orientation preserving homeomorphic extensions to the whole plane. More precisely, we prove that this group is the uniform closure of the group generated by the Thompson group of the unit circle and an inversion. Moreover, this result is quantitative in the sense that distortions of the approximating maps are uniformly controlled by the distortion of the given map.
}
\endabstract

\maketitle

\section{Introduction}\label{S:Intro}
\no
Quasisymmetric geometry of fractal sets has attracted  substantial interest in recent years.
A natural invariant in this category is the group of quasisymmetries of the set.
One can roughly classify such a set as ``little quasisymmetric'' or ``highly quasisymmetric'', depending on whether
this group is finite or infinite. We are interested in this dichotomy for Julia sets of rational maps.
In our previous paper~\cite{BLM}, 
joint with Mario Bonk,  we described a class of little quasisymmetric  Julia sets that are Sierpi\'nski carpets. 
The goal of this paper is to give an example of a highly quasisymmetric Julia set, the  {basilica}
(see Figure~\ref{F:QJ}), 
and to describe its group of quasisymmetries. 
%
%In this paper we consider the quadratic map $f(z)=z^2-1$; 
%\ see Figure~\ref{F:QJ} for its filled-in Julia set $\Ju$. This Julia set is referred to as \emph{basilica}.
%
%\comm{
\begin{figure}
[htbp]
\begin{center}
\includegraphics[height=40mm]{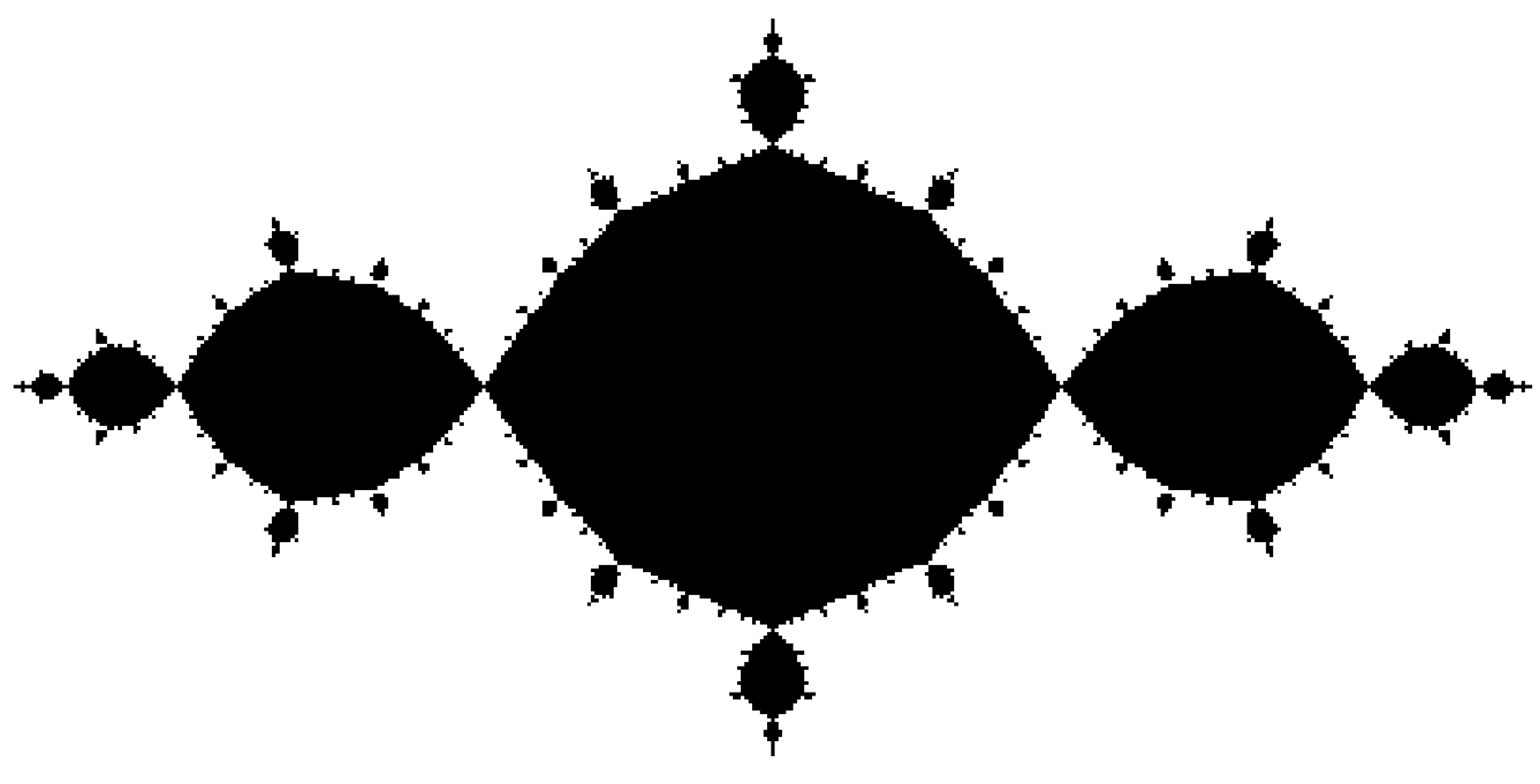}
\caption{
The filled Julia set of $f(z)=z^2-1$.
}
\label{F:QJ}
\end{center}
\end{figure}
%}

To formulate the main result,
let us give quick definitions of the main objects, referring to the main body of the paper for 
their precise versions. 

%The \emph{Fatou set} of a rational map $f$ is the set of all points in the sphere $\hC$ that possess neighborhoods 
%where the family of all forward iterates of $f$ is normal. The complement of the Fatou set of $f$ in $\hC$ 
%is the \emph{Julia set} of $f$; we denote it by $\Ju$. The Julia set of every rational map is non-empty and compact.
%We refer the reader  to \cite{Be, Mi, St} for background on iteration of rational maps and related notions.

Let  $f\: \C\ra \C$ be a polynomial of degree $\geq 2$.
Its \emph{filled Julia set} $\KK(f)$
is defined as the set of non-escaping points, 
and the {\em Julia set} $\Ju(f)$ is defined as the boundary of $\KK(f)$.  

The \emph{basilica} is the filled Julia set of the quadratic polynomial $f\: z\mapsto z^2-1$.%
\footnote{In this paper, we refer informally to the corresponding Julia set $\Ju$ also as ``basilica''.} 
This polynomial has a  superattracting cycle $\bga = \{0,-1\}$ of period two,  
and $\inter \KK(f)$ coincides with the {basin} of this cycle. 
By definition, the \emph{immediate basin} of $\bga$ is the union of two components $U_0$ and $U_{-1}$
of the basin containing $0$ and $-1$, respectively.
These components are Jordan discs, and the Riemann mapping $\phi_0\: U_0\ra \D$, such that $\phi_0(0)=0,\ \phi_0'(0)<0$, brings the return map 
$f^2\: U_0\ra U_0$ to the monomial form $g\: z\mapsto z^2$.

Given a  homeomorphism  $\eta\: [0,\infty)\to[0,\infty)$, 
a homeomorphism $h\: \Ju(f)\ra \Ju(f)$ is called a \emph{quasisymmetry} 
with \emph{distortion function} $\eta$ (or $\eta$-\emph{quasisymmetry}) if 
$$
\frac{|h(u)-h(v)|}{|h(u)-h(w)|}\leq\eta
\bigg(\frac{|u-v|}{|u-w|}\bigg),
$$
for all triples of distinct points $u,v$, and $w$ in $\Ju(f)$.
In fact, this definition is applicable to a homeomorphism $h\: X\ra X$ of any metric space $X$, with the Euclidean distance replaced by the corresponding metric.
Moreover, in the case of the complex plane, $X=\C$, 
$\eta$-quasisymmetry is equivalent to $K$-quasiconformality, quantitatively
(provided the map is normalized at two points).
We say that a homeomorphism of $\Ju(f)$ is \emph{topologically extendable} if it has an extension to an orientation preserving homeomorphism of $\C$.

%We say that $h\: \Ju\ra \Ju$ is an \emph{ambient $K$-quasisymmetry} if $h$ is a quasisymmetry that admits a
%$K$-qc extension to the whole complex plane. If the value of the constant $K$ is irrelevant, we suppress it and say  that $h$ is an \emph{ambient quasisymmetry}.

The \emph{Thompson group} $T$  is defined as the group of piecewise linear
(in the angular coordinate)  homeomorphisms of the unit circle $\T$ with breaks at some dyadic points and slopes
equal to integer powers of $2$. By means of the Riemann mapping $\phi_0$, we can make this group act on $\di U_0$.
In this paper, we show that this action admits an extension to an action by  quasisymmetries of the basilica Julia set $\Ju\equiv\Ju(f)$.
We also construct one more quasisymmetry of $\Ju$, an involution $\iota$, which permutes the components
$U_0$ and $U_{-1}$ of the immediate basin.
Let $\hat T$ be the \emph{extended Thompson group} of quasisymmetries of $\Ju$ generated by $T$ and $\iota$.
Our main result asserts that this group generates quantitatively the whole group of topologically extendable  quasisymmetries of $\Ju$:
%Our main result is the following description of the group of ambient quasisymmetries of $Ju$:

\begin{theorem}\label{thm:thom}
For any distortion function $\eta$ there exists a distortion function $\eta'$, such that  
for any topologically extendable $\eta$-quasisymmetry $\xi$ of the basilica Julia set $\Ju$ 
there exists a sequence of $\eta'$-quasisymmetries $\tau_n$ of $\Ju$ that belong to the 
extended Thompson group $\hat T$ and 
uniformly converge to $\xi$.
\end{theorem}
%
% quasisymmetric homeomorphisms of $\Ju$ that have quasiconformal extensions to the whole complex plane $\C$, 
% in terms of Thompson's group $T$ and an inversion $\iota$, defined in Sections~\ref{S:TG} and~\ref{S:Inv}, respectively.
%For definitions and basic properties of quasiconformal and quasisymmetric maps see Section~\ref{S:Qsom}.
%
%\begin{theorem}\label{thm:thom}
%Let $\xi$ be a $K$-quasiconformal homeomorphism of the plane that preserves $\Ju$ setwise.  
%Then there exists   a homeomorphism $\eta\: [0,\infty)\to[0,\infty)$ that depends only on $K$, such that the following holds. For every $%\epsilon>0$, there exists a homeomorphism $\tau$ of $\Ju$ that belongs to the group generated by $T$ and $\iota$, that is $\eta$-quasi%symmetric, and such that 
%$$
%|\xi-\tau|<\epsilon,
%$$ 
%where $|\xi-\tau|$ denotes the uniform distance on $\Ju$ between $\xi$ and $\tau$.
%\end{theorem}

Thus, the basilica is highly quasisymmetric.
As we have already mentioned, this 
contrasts with the result of  \cite{BLM} concerning Sierpi\'nski Julia sets.
Note that the {standard} Sierpi\'nski carpet is little quasisymmetric as well \cite{BM}.
On the other hand, it was shown in \cite{Me} that the ``slit carpet''
is highly quasisymmetric,  and the corresponding group of quasisymmetries bears some
similarity with the Thompson group.

Note that the limit sets of non-elementary Kleinian groups are {highly symmetric},
as they admit infinite groups of M\"obius symmetries. 
Let us say that a compact subset $K\subseteq\hC$ of the sphere is \emph{rigid} if all its quasisymmetries are
M\"obius. It was shown in \cite{BKM} that   ``Schottky sets'' of zero area are rigid,
contrasting again with our result. 

Let us also mention a recent work of J.~Belk and B.~Forrest \cite{BF} who studied 
a Thompson-like group of circle homeomorphisms that preserve the invariant lamination of the basilica, 
and hence descend to homeomorphisms of the basilica itself.
Our result implies that these homeomorphisms are quasisymmetric
(and, in fact, form a dense subgroup of basilica quasisymmetries). 
Unlike \cite{BF}, we approach the problem from ``inside'' of the basilica,
beginning with the Thompson group acting on  the immediate basin.

Let us also note that 
a minor modification  of our construction (that does not need Lemma~\ref{L:ConfElev}) would produce the following result. If $\xi$ in Theorem~\ref{thm:thom} is only assumed to be a homeomorphic self-map of $\Ju$ that has an orientation preserving homeomorphic extension to $\C$, then there is a sequence of maps $\tau_n,\ n=1,2,\dots$, in $\hat T$ that converge uniformly on $\Ju$ to $\xi$. 
So we can describe the full group of extendable homeomorphisms of $\Ju$ as the uniform closure of $\hat T$. However, even though each map $\tau_n$ is quasisymmetric, one generally does not have a uniform control on their distortion functions. The key technical step of our construction is Lemma~\ref{L:ConfElev} that provides us with uniform quasisymmetric control of approximating maps.

Let us finally note that our method extends in a straightforward way to 
hyperbolic Julia sets in the ``main molecule'' of the Mandelbrot set
(obtained from $z^2$ through a finite cascade of satellite bifurcations),
e.g., to the  Douady rabbit.

\subsection{Notation and terminology}\label{SS:Not} Throughout the paper we denote $\N\cup\{0\}$ by $\N_0$. 
We let $\D$ be the open unit disc in the complex plane $\C$, and let 
% $\C\setminus{\overline\D}$ its complement. We denote 
 $\T=\di \D$  be the unit circle in $\C$.
% which is the common boundary of $\D$ and $\C\setminus{\overline\D}$.

In what follows, we  label a point $z=e(\theta) = e^{2\pi i \theta} \in \T$ by $\theta\in \R/\Z$. 
%and  refer to $\theta$  as the ``argument'' of $z$.
In other words, the angular measure on the circle $\T$ is scaled so that its total length is equal to 1,
which is standard in dynamics.
%In what follows, we  label a point on the unit circle $\T$ whose argument is $\theta$ by $\theta/2\pi$
% We also scale angular lengths of intervals  in $\T$ by $2\pi$, so that the total length of $\T$ is 1. 
Points 
% $e( k /2^n)\in \T$, $k=0,1,\dots 2^n-1,\ n\in\N_0$,
 $e( (2k+1) /2^n)\in \T$, $k=0,1,\dots, 2^{n-1}-1,\ n\in\N$,
are called \emph{dyadic points} \emph{of level $n$}, and $e(0)$ is \emph{the dyadic point of level} $0$.
% The set of dyadic points in $[0,1]$ is the union of all points of the form $k/2^n$, 
%for  $n\in\N_0$ and $k\in\N_0,\ 0\le k\le 2^n$. The set of dyadic points on the unit circle $\T$ is the set of points in $\T$  
%whose arguments have the form $2\pi k/2^{n}$ for  $n\in\N_0$ and $k\in\N_0,\ 0\le k< 2^n$. 
Taking the union of dyadic points of levels $m\leq n$, we obtain $2^n$ points
$e(l/2^n)$, $l=0,1,\dots, 2^n-1$, that tessellate the circle  
into $2^n$ (closed) intervals $I^n_l$. We will refer to these intervals as
\emph{dyadic intervals of level $n$}. 

Let $g$ denote the map $z\mapsto z^2$,
which is \emph{ doubling} in the angular coordinate of $\T$, i.e.,  
$\theta\mapsto 2\theta\ \mod \Z$.  
The dyadic points  $z\in \T$ are % 1 and 
dynamically identified as  the iterated preimages of the fixed point $1$, %  under all the iterates of the map $g$. 
with the level equal to the smallest $n\in\N_0$ such that  $g^n(z)=1$.
%We say that a dyadic point $z\in \T$ is at \emph{level} $n\in\N_0$ if $g^n(z)=1$. 
%and $n$ is the smallest number with this property. 

\section{Dyadic subdivision of bounded Fatou components}\label{S:Dyad}
 \no
The reader can consult \cite{CG,Mi} for a general introduction to the  iteration theory of rational functions,
and  \cite{DH,L} for particular features of the dynamics of quadratic polynomials.  

The \emph{Fatou set}  is the complement of the Julia set, $\FF(f)=\C\sm \Ju(f)$.  
Its connected components are called \emph{Fatou components}. 
Bounded Fatou components can also be identified as the connected components of  the interior of the filled Julia set $\KK(f)$. 

As we have already mentioned, 
the \emph{basilica map} $f(z)=z^2-1$ is specified by the property that it has a superattracting periodic cycle $\{0,-1\}$ of period two. 
In particular, it is a postcritically finite {hyperbolic} map. 
We let $U_0$ and $U_{-1}$ be the Fatou components of $f$ that contain $0$ and $-1$, respectively.
Both of these components are Jordan discs.
The map $f$ takes $U_0$ onto $U_{-1}$ as a double branched covering, 
and it takes $U_{-1}$ back  onto $U_0$ conformally. 

Let $\phi_0\: (\overline{U_0},0, \alpha) \ra (\overline{\D},0,1)$ be the {\em B\"ottcher coordinate} of $U_0$
(which coincides with the appropriately normalized Riemann uniformization of  $U_0$),
homeomorphically extended  to the boundary.
% The second iterate $f^2$ of $f$ fixes $U_0$. The B\"ottcher coordinate $\phi_{0}$ 
It conjugates the return map $f^2\: \overline{U_0}\ra \overline{U_0} $
 to the monomial map $g\: \overline{\D} \ra \overline{\D}$, $z\mapsto z^2$. 
%of the unit disc. 

The left-most point of the closure $\overline{U_0}$ on the real line is a fixed point  $\alpha$ of $f$, 
is equal to $(1-\sqrt{5})/2$. It is also the unique point of intersection of $\overline{U_0}$ and $\overline{U_{-1}}$, 
% The fixed point $\alpha$
called the \emph{root} of each of these components. 
 Moreover, $\alpha$ is a \emph{global cut-point} for the Julia set $\Ju$:
puncturing  $\alpha$ out results in  breaking  $\Ju$ into two % non-empty 
connected components. 
All the preimages of $\alpha$ under the iterates of $f$ are therefore also global cut-points.
The other fixed point of $f$ is $\beta=(1+\sqrt{5})/2$. 
It is the right-most point of the intersection of $\Ju$ with the real line. 

Every bounded Fatou component $U$ of $f$ eventually, i.e., under a certain iterate of $f$, lands in the cycle $\{U_0, U_{-1}\}$. 
% If $U$ is any  bounded Fatou component of $f$, 
In fact, for each $U$ there exists a unique $n\in\N_0$ such that $f^n\: U\to U_0$ is a conformal map. 
We call such $n$ the \emph{dynamical distance} from $U$ to $U_0$.
It follows that all bounded Fatou components of $f$ are Jordan discs as well.  
% i.e., they are homeomorphic to the unit disc under global homeomorphisms of the plane.  
Then the map $\phi_{U}=\phi_{0}\circ f^n$ is a conformal map of $U$ onto $\D$
that extends to a homeomorphism $\overline{U}\ra \overline{\D}$. It is called the {\em B\"ottcher coordinate} of $U$. 
The \emph{root} $\alpha_U\in \di U$ of $U$ is defined as the preimage  of the root  $\alpha\in U_0$ under $f^n$,
or equivalently, as the point in $\di U$  whose B\"ottcher coordinate  $\phi_U(\alpha_U)$ is equal to $1\in \T$.

Similarly, we define \emph{dyadic points} on $\di U$ as points whose B\"ottcher coordinates are dyadic.
Dyadic points of level $\leq m$ induce a \emph{ level $m$ dyadic subdivision} of $\di U$. 
%
%The subdivision of $\T$ by the dyadic points  induces a \emph{dyadic subdivision} of $\dee U$ via the map $\phi_U$. 
%The \emph{root} of a bounded Fatou component $U$ is the point of the dyadic subdivision that is mapped 
%to $\alpha$ under the conformal map $f^n\: U\to U_0$, where $n$ is the dynamical distance from $U$ to $U_0$. 
%The root of $U$ is therefore the preimage of 1 under the B\"ottcher coordinate $\phi_U$.
% We say that a dyadic point of $\dee U$ is \emph{at level} $n$ if it corresponds to a level $n$ dyadic point on $\T$ under $\phi_U$.
%
It can also  be described as follows: 
%If we take preimages of $\alpha$  under $f$, we obtain dyadic decompositions of all bounded Fatou components. 
for a  bounded Fatou component $U$ of dynamical distance $n\ge0$ from $U_0$ 
and for an integer number $m\ge0$, 
the full preimage $f^{-(n+2m)}(\alpha)$ gives the dyadic decomposition of $\dee U$ of level $m$ that consists of $2^{m}$ points. 
%(Here, if $n>m$, then the corresponding set is empty.) 
% In B\"ottcher coordinates these correspond to the full preimage of 1  under the $(m-n)$th iterate of the map $g(z)=z^2$.

The basilica filled Julia set $\KK\equiv \KK(f)$   has the following combinatorial structure of a (non-locally finite) tree  $\TT$.  
The vertices of $\TT$  are the bounded Fatou components of $f$.  
(In what follows we make no distinction between bounded Fatou components and the vertices of $\TT$.)
Two vertices $U$ and $V$  are connected by an edge $E$ if and only if 
% the corresponding components 
they \emph{touch}, 
i.e., their closures intersect. 
Let $\de(U)$ stand for the {combinatorial distance} in this tree from a vertex $U$ to the \emph{main vertex} $U_0$. 
If two vertices $U$ and $V$ are adjacent, then $|\de (U) - \de(V)|=1$.   
Under these circumstances,  if $\de(U) = \de(V)+1$ then $V$ is called  
the \emph{principal} vertex of $E$, while $U$ is called a \emph{satellite} of $V$. 
%Moreover, this tree has the following local structure. 
%%%Each vertex of the tree corresponds via the B\"ottcher coordinate to a copy of the unit disc $\D$ 
%%%with marked dyadic points on its boundary $\partial\D$.  
%An edge of $\mathcal T$ corresponds to a dyadic point on the boundary of the vertex $U$ with smaller dynamical distance to the central 
%component $U_0$; see below for precise definitions and justifications. 
%
%The maps $\phi_U$ give the correspondence between the vertices of $\Tee$ and copies of $\D$. 
%
%Now we turn our attention to edges. 

The \emph{main edge} of $\TT$, denoted $E_0$, is the edge that connects $U_0$ to  $U_{-1}$
via their common root $\alpha$. 
%This edge corresponds to the pair of roots of $U_0$ and $U_{-1}$. 
%
The higher level  edges are described as follows:

\begin{lemma}\label{L:Basepoint}
Let $U$ and $V$ be distinct bounded Fatou components that touch at a point $z$. 
Let $n\in\N_0$ and $m\in\N_0$ be the dynamical distances from $U$ and $V$ to $U_0$, respectively, with $n \geq m$. 
Then $U$ is a satellite of $V$ 
% Then the combinatorial distance from $V$ to $U_0$ in $\Tee$ is one less than that from $U$ to $U_0$, 
and the root  of $U$ is $z$. 
\end{lemma}

\no
\emph{Proof.} 
The distances $m$ and $n$ cannot be equal  because otherwise we get a contradiction with the fact that $f^m$ is conformal in a neighborhood of $z$. 
Application of $f^k,\ k\leq m$,  preserves the differences in dynamical and in combinatorial distances from $U$ and $V$ to $U_0$, respectively.  This implies that the combinatorial distance from $V$ to $U_0$ in $\TT$ is one less than that from $U$ to $U_0$.   Now,  $f^m(U)$ is a bounded Fatou component that touches $f^m(V)=U_0$. The point of the intersection is $f^m(z)$, and $n-m\in\N$ is the dynamical distance from $f^m(U)$ to $U_0$.  Thus $f^{n-m-1}(f^m(U))$ is either $U_{-1}$ or $U_1$, the bounded Fatou components that contain $-1$ and 1, respectively. In either case we must have  $f^{n-m-1}(U_0)=U_0$ because $f$ is at most two to one in each Fatou component. Hence $f^{n-1}(z)=f^{n-m-1}(f^m(z))$ is either $\alpha$ or $-\alpha$, and thus $f^{n}(z)=\alpha$, i.e., $z$ is the root of $U$. 
\qed

\medskip

% The point $z$ is mapped by $\phi_{V}$ to a dyadic point $d\in\partial \D$.  Therefore, 
Thus, every edge $E$ 
%  has a dyadic representation 
can be labeled by the principal component $V$ and 
a dyadic number $d\in \T$ representing the point on $\di V$ where the satellite component $U$ touches $V$.    
% the boundary of  the Fatou component with the smaller dynamical (and combinatorial) distance to $U_0$. 

\section{
Dynamical partitions of the Julia set $\Ju$}\label{S:Lam}
\no
Let $U_\infty$ denote the \emph{basin at infinity}, i.e., the unbounded Fatou component of $f$, and
let $\phi_\infty$ be the \emph{B\"ottcher coordinate} of $U_\infty$. 
Namely, $\phi_\infty$ is the conformal map of $U_\infty$ onto the complement of the closed unit disc $\C\setminus{\overline\D}$ 
% such that the map $f$ is conjugate by $\phi_\infty$ to $g(z)=z^2$. 
conjugating $f$ to $g\: z\mapsto z^2$.

The {\em external ray} $\rho^\theta$ with angle $\theta\in \R/\Z$ is defined as  the pullback of the straight ray
$\rho^\theta= \{ r e(\theta)\:   1< r<\infty \} $ under the B\"ottcher map $\phi_\infty$. 
The external rays form an invariant foliation with $f$ acting by the angle-doubling:
$f(\rho^\theta) = \rho^{2\theta}$. 

It is known that the Julia set of a hyperbolic map $f$ is {locally connected},
so the inverse B\"ottcher map  
$ \inv_\infty  = \phi_\infty^{-1}\: \D\ra U_\infty$ extends continuously to the boundary, and so induces a continuous boundary map
of $\T$ onto $\Ju$. It follows that any ray $\rho^\theta$ {\em lands} at some point 
$z_\theta\in \Ju$.   For a circle  arc $I\subseteq \T$, we call
$$
      \inv_\infty (I) = \{ z_\theta\:  e(\theta)\in I \}  
$$ 
a {\em Julia arc}.

In the case of the basilica, there exist exactly two rays, $\rho^{\pm 1/3}$, that land at the fixed point $\alpha$.
They bound two (open) sectors, $S_0\supset U_0$ and $S_\alpha\supset U_{-1}$.
The latter sector is also called the \emph{wake  rooted at}  $\alpha$. 
The intersection $ \KK_\alpha:= (\KK\cap S_\alpha)  \cup \{ \alpha \} $
(with the added root)  is
called the \emph{limb} of $\KK$ \emph{rooted at} $\alpha$. 
We denote by $\Ju_\alpha$ the Julia arc $\dee \KK_\alpha$. 

This picture can be spread around to the iterated preimages of
$\alpha$.
Let  $z\in \Ju$ and  $f^n(z) = \alpha$ for some $n\in \N_0$.
We choose  the smallest moment $n$ like this. Then $z$ is the landing point
of exactly two rays,  $f^n$-preimages of $\rho^{\pm 1/3}$. 
They bound the unique (open) sector $S_z$ that does not contain $U_0$;
it is called the \emph{wake rooted at}  $z$. 
Moreover, if $n\geq1$, then $f^n$ conformally maps $S_z$ onto $S_0$. 
Let $\KK_z: = (S_z\cap \KK)\cup \{z\}$ be the \emph{limb} of $\KK$
\emph{rooted} at $z$, and  let $\Ju_z:= \dee \KK_z$ be the corresponding Julia arc. Below we may also refer to $\Ju_z$ as a \emph{limb rooted at} $z$.

Let $\mathcal R_n,\ n\in\N_0$, be the family of all  external rays in $U_\infty$ 
that land at points of the full preimage $f^{-n}(\alpha)$. 
Note that each external ray of $\mathcal R_n$ lands on the boundaries of exactly two bounded Fatou components,
adjacent components of the tree $\TT$. 

We say that distinct external rays $\rho_1$ and $\rho_2$ of $\mathcal R_n$ are  \emph{adjacent} 
if they are not separated in $U_\infty$  by other external rays of $\mathcal R_n$
(in other words, the angles $\theta_1, \theta_2$ are adjacent points of the set $g^{-n}\{\pm1/3\}\subset\T$).

\begin{lemma}\label{L:Adj}
If $\rho_1, \rho_2\in\mathcal R_n$ are adjacent external rays, then there exists a bounded Fatou component $U$ of $f$ such that $\rho_1$ and $\rho_2$ land at boundary points of  $U$. 
\end{lemma}

\no
\emph{Proof.}
We apply induction. The statement is trivially true for $n=0$. Assuming that the statement is true for $n-1$, where $n\ge1$, let $\rho_1$ and $\rho_2$ be two adjacent external rays of $\mathcal R_n$. If we apply $f$ to $\rho_1$ and $\rho_2$, we obtain two external rays of $\mathcal R_{n-1}$ that are necessarily adjacent. Otherwise, we would apply $f^{-1}$ to the region in $U_\infty$ between $f(\rho_1)$ and $f(\rho_2)$ that contains a separating external ray to get a contradiction. Note that in the B\"ottcher coordinate of $U_\infty$ the arc on $\T$ that corresponds to the landing points of $\rho_1$ and $\rho_2$ and does not contain the landing points of other rays in $\mathcal R_n$ has angular length at most  $1/3<1/2$, and hence $f$ is conformal in the region of $U_\infty$ between $\rho_1$ and $\rho_2$ that does not contain other external rays of $\mathcal R_n$.
Therefore, $f(\rho_1)$ and $f(\rho_2)$ land on the boundary of the same bounded Fatou component $U$. If $U=U_{-1}$, then $\rho_1$ and $\rho_2$ must land on the boundary of $U_0$, because $U_0$ is the only preimage of $U_{-1}$ under $f$.
If $U\neq U_{-1}$, any branch of the map $f^{-1}$ restricted to $U$ is conformal onto another bounded Fatou component $V$, and there are exactly two such branches because $f$ has degree 2. Suppose $\rho_1$ lands on the boundary of $V$. The branch $f^{-1}$ that takes $U$ onto $V$ extends conformally across the arc $\omega$ on the boundary of $U$ between the landing points of $f(\rho_1)$ and $f(\rho_2)$ into the region $R$ in $\C\setminus \overline U$  that satisfies the following properties. The arc $\omega$ does not contain other landing points of external rays in $\mathcal R_{n-1}$, and the region $R$ is bounded by $\omega$, $f(\rho_1)$, and $f(\rho_2)$. The image of $f(\rho_2)$ under this branch of $f^{-1}$ is then necessarily the external ray $\rho_2$, and we are done.    
\qed

\medskip

For $n\in \N_0$, let:

\smallskip\noindent -- $D_n$ be the $n$-fold preimage $g^{-n} \{\pm 1/3\} \subset \T$ (consisting of  $2^{n+1}$ points);

\smallskip\noindent
--  $P_n$ be the tiling of the circle $\T$ by the points of $D_n$
      (comprising $2^{n+1}$ closed arcs $I_k\subset \T$);

\smallskip\noindent
--    $\Pi_n= \inv_\infty  (P_n)$ be the corresponding tiling of $\Ju$
     (comprising $2^{n+1}$ closed  Julia arcs  $J_k= \inv_\infty (I_k)\subset \Ju$
     with \emph{endpoints} at preimages of the fixed point $\alpha$ of \emph{level} $\leq n$).
%The external rays in $\mathcal R_n,\ n\ge 0$, induce a partition $\Pi_n$ of the Julia set $\Ju$ into 
%$2^{n+1}$ closed connected sets $J_k,\  k=1,2,\dots, 2^{n+1}$,  where 

\smallskip
Each set $J_k$ corresponds to a pair of adjacent external rays of $\mathcal R_n$, as follows. 
Let $\rho_1$ and $\rho_2$ be two adjacent external rays of $\mathcal R_n$ whose landing points are $z_1$ and $z_2$, respectively. 
There are two cases that need to be considered: either $z_1=z_2$ or $z_1\neq z_2$. 

\smallskip
\no
\emph{Case 1.}
If $z=z_1=z_2$, this point is the root of a unique bounded Fatou component $U$ unless $z=\alpha$, the fixed point of $f$. This follows from Lemma~\ref{L:Basepoint}. 
If $z=\alpha$, we choose $U=U_{-1}$.
The Julia arc $J_k$ of $\Pi_n$ that corresponds to $\rho_1, \rho_2$ is $\Ju_z$.

\smallskip
\no
\emph{Case 2.}
Now assume that $z_1\neq z_2$ and let $U$ be a bounded Fatou component of $f$ whose boundary contains $z_1$ and $z_2$.  Such $U$ exists by Lemma~\ref{L:Adj}.  
Let $\omega$ be the arc on $\dee U$ between $z_1$ and $z_2$ that contains no other landing points of external rays in $\mathcal R_n$. Let $R$ be the region in $\C\setminus\overline U$ that is bounded by $\omega$, $\rho_1$, and $\rho_2$. The Julia arc $J_k$  that corresponds to the pair $\rho_1, \rho_2$ is the closure of the intersection of $\Ju$ with the region $R$.     

\smallskip
It follows immediately from the definition that the map $f^n$ takes each Julia arc $J_k$ of $\Pi_n$ onto the closure of one of the two connected components of  $\Ju\setminus\{\alpha\}$: either $\Ju_0$ that contains the boundary of $U_0$ or $\Ju_\alpha$ that contains the boundary of $U_{-1}$. The first case occurs when the landing points $z_1$ and $z_2$ are the same. In this case the map $f^n$ is one to one. In the other case, i.e., $z_1\neq z_2$, the preimage of each point $z$ of $\Ju_{\alpha}$ under $f^n$ is a singleton except if  $z=\alpha$, when the preimage consists of two points.

% Under the map $\phi_\infty$ the two external rays that land at $\alpha$ have external angles $1/3$ and $2/3$. 
% The inverse $\phi_{\infty}^{-1}$ has a continuous extension to the boundary $\T$, still denoted by $\phi_{\infty}^{-1}$.
% It is not a homeomorphism however. 
%For example, both points ${1/3}$ and  ${2/3}$ on $\T$ are mapped to $\alpha$.  
%The landing points of external rays in $\mathcal R_n,\ n=0,1,2,\dots$, 
%correspond via $\phi_\infty$ to the full preimage set of $\{1/3,2/3\}$ under $g^n$. We denote this preimage set in $\T$ by $D_n$.

\begin{lemma}\label{L:OuterDyadSubdiv}
The set $D_n,\ n\ge 0$, is the set of $2^{n+1}$ points on the unit circle such that the angular lengths of complimentary intervals alternate between $1/(3\cdot2^{n})$ and $2/(3\cdot2^{n})$. In particular, there exists a constant $L\ge1$ independent of $n$, and an orientation preserving  piecewise linear $L$-bi-Lipschitz map $\psi$ of $\T$ whose break points are points of $D_n$, and
such that $\psi(D_n)$ is a set of $2^{n+1}$ points on the unit circle such that all complementary intervals have equal lengths.  
\end{lemma}
\no
{\emph{Proof.}}
Let $I$ and $J$ be two adjacent complementary intervals at level $n$. Application of the map $g(z)=z^2$ doubles the lengths of each one of them and keeps them adjacent. Because for $n=0$ the points $1/3$ and $2/3$ satisfy the desired property and form a 2-cycle under the dynamics of $g$, the first part of the claim follows from induction. The second part follows from the observation that in order to make all complementary intervals to have the same length, one needs to scale the complementary intervals of length $1/(3\cdot 2^n)$ by $3/2$ and the complementary intervals of length $2/(3\cdot 2^n)$ by $3/4$. In particular, $L=3/2$. 
\qed

\medskip

Let us finish this section with a description of  a \emph{pinched disk} topological model 
for the basilica.  Let $d$ be the diameter of $\D$ with one endpoint
at $e(1/3)$, and let $\D_\pm$ be the corresponding semi-disks. For the sake of definiteness we assume that $e(2/3)$ is contained in ${\D_-}$.
Let us connect the points  $e(1/3)$ and $e(2/3)$ of $D_0$
with the hyperbolic geodesic $\gamma^0\subset \overline\D$. 
The points of $D_1$ split into two symmetric pairs, one contained in $\overline{\D_+}$, 
the other contained  in $\overline{\D_-}$. 
Each of these pairs can be connected with a hyperbolic geodesic. 
Since $\{e(1/3),e(2/3)\}$ is a periodic cycle for the map $g$, one of these geodesics, namely the one contained in $\overline{\D_-}$, is the geodesic $\gamma^0$.
We denote the other geodesic by $\gamma^1$.  
The full preimage 
of $\di \gamma^1$ under the doubling map $g$ consists of two pairs of points in $D_2$,
 one contained in $\D_+$,  the other contained  in $\D_-$. 
 Connecting each of these pairs with a hyperbolic  geodesic, we obtain two new geodesics in $D_2\setminus D_1$, denoted by
 $\gamma^2_{+}$ and $\gamma^2_{-}$. 
By the same procedure,
for each geodesic $\gamma^2_\epsilon,\ \epsilon=\pm$, 
we obtain two new geodesics (``pullbacks'' of  $\gamma^2_\epsilon$), 
one in $\D_+$ and the other contained  in $\D_-$, which we denote by $\gamma^3_{(\eps_1\, \eps_2)},\ \eps_i\in \pm,\ i=1,2$.
Proceeding this way, for each level $n\in \N$, we can construct $2^{n-1}$ hyperbolic geodesics
$\gamma_{\bar \eps}^n$, $\bar\eps = (\eps_1\dots \eps_{n-1})$, $\eps_i\in \pm$,  
paring the points of $D_n\setminus D_{n-1}$.   It is easy to show that all  geodesics $\gamma_{\bar \eps}^n$ are disjoint.  
One can also show that together they form a closed subset $Q$ of $\overline \D$.
% Taking the closure of this family of geodesics, we obtain a closed set $Q\subset \overline\D$
% which is decomposed into the union of disjoint  hyperbolic geodesics.
This subset (endowed with a partition into the geodesics $\gamma_{\bar \eps}^n$) is called the \emph{basilica lamination}; see Figure~\ref{F:BL}.  

%\comm{
\begin{figure}
[htbp]
\begin{center}
\includegraphics[height=60mm]{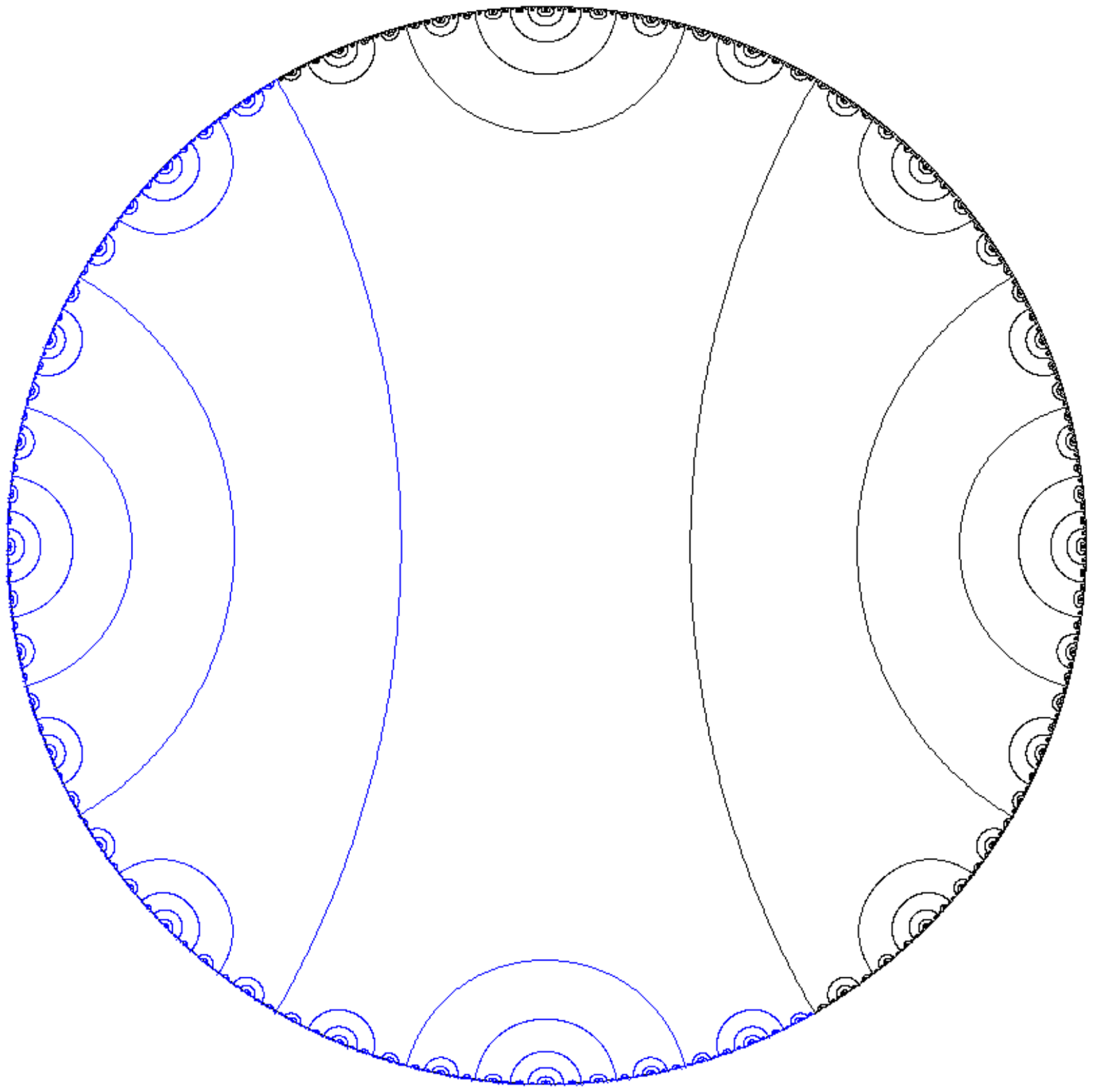}
\caption[]{
The basilica lamination\protect\footnote{Source: https://commons.wikimedia.org/wiki/File:Basilica\_lamination.png}.
}
\label{F:BL}
\end{center}
\end{figure}
%}

Let us consider an equivalence relation $\underset{f}{\sim}$ on $\C$ whose classes are either the geodesics
of the basilica lamination or single points.

\begin{theorem}{\rm (see, e.g., \cite[Theorem~24.33]{L})}\label{model thm} 
The quotient of $(\C, \overline\D )$ by the equivalence relation $\underset{f}{\sim}$
is homeomorphic to $(\C, \KK)$. Moreover, this homeomorphism coincides with the
inverse B\"ottcher coordinate $\psi_\infty$ on $\C\sm \overline\D$ and sends the geodesics
$\gamma^n_{\bar \eps}$ to the corresponding  level $n$ preimages of the $\alpha$-fixed point.   
\end{theorem}

The equivalence relation  $\underset{f}{\sim}$ induces an equivalence relation
on the circle $\T$ that pairs points of $D_\infty:= \bigcup D_n$, where
$\theta_1 \underset{f}{\sim} \theta_2$ if and only if the rays
$\rho^{\theta_1}$ and $\rho^{\theta_2}$ land at the same point of $\Ju $.
We will refer to it as the \emph{basilica lamination of} $\T$.
(This ``lamination'' has zero-dimensional leaves.)
Obviously, the basilica lamination of $\T$  contains the same amount of
information as the basilica geodesic lamination of $\overline\D$.
% so we will feel free to identify these two  objects.  

For the sake of reference, let us state a simple lemma:

\begin{lemma}\label{homeos}
Any homeomorphism $\xi\: U_\infty\cup \Ju \ra U_\infty\cup \Ju$ lifts (via the extended inverse B\"ottcher coordinate $\psi_\infty\: \C\sm \D \ra U_\infty\cup \Ju$) 
to a homeomorphism $h\: \C\sm \D \ra \C\sm \D$ that preserves the basilica lamination  of $\T$, and whose restriction $\xi_\infty$ to $\T$ satisfies
\begin{equation}\label{E:SC}
\inv_\infty  \circ\xi_\infty=\xi\circ\inv_\infty .
\end{equation}
Conversely, any  homeomorphism $h\: \C\sm \D \ra \C\sm \D$ that preserves the basilica lamination of $\T$
descends to a homeomorphism  $\xi\: U_\infty\cup \Ju \ra U_\infty\cup \Ju $.   
\end{lemma}

\begin{proof}
  On  $\C\sm \overline\D$, the map $h$ is defined as $\phi_\infty \circ \xi\circ    \psi_\infty$. 
By the Carath\'eodory Theory, the closure $\C\sm \D$ is naturally homeomorphic
to the prime end compactification $\cl^C U_\infty$  of $U_\infty$.  The definition of prime ends easily implies that 
any homeomorphism $\xi\: U_\infty\cup \Ju \ra U_\infty\cup \Ju $ induces a homeomorphism of $\cl^C U_\infty$.
It follows that $h$ extends continuously to the unit circle $\T$. 
This extension respects the basilica lamination of $\T$ since $\xi$ respects the
landing pairing between the rays. Equation~\ref{E:SC} follows by
equating the boundary values of $\psi_\infty\circ h$ and
$\xi\circ\inv_\infty $.

Conversely,  any homeomorphism   $h\: \C\sm \D \ra \C\sm \D$ that preserves the basilica lamination on $\T$
descends to a homeomorphism of the quotient $  (\C\sm \D)/  \underset{f}{\sim} $. By Theorem \ref{model thm},
the latter is naturally  homeomorphic to $U_\infty \cup \Ju$, providing a desired homeomorphism $\xi$.     
\end{proof}

See \cite{Th,D,L} for a detailed discussion of geodesic laminations and pinched models. 

\section{Local properties of quasisymmetries of $\Ju$}\label{S:Qsom}
\no
Before we proceed, we briefly recall basic definitions and facts on quasiconformal and quasisymmetric maps. For more background one can consult~\cite{Va, AIM, He}. 
  
 A homeomorphism $f \: U \to \tilde U$ between open regions in the plane $\C$ or the Riemann sphere $\hC$ is called \emph{quasiconformal} if $f$ is   in the Sobolev space $W_{\rm loc}^{1,2}$
and if there exists a constant $K\ge 1$ such that the (formal) Jacobi   matrix $Df$
satisfies
%\begin{equation}\label{eq:defqr}
$$
||Df(z)||^2\leq K\det(Df(z))
%\end{equation}
$$
for almost every $z \in U$.  In this case we say $f$ is $K$-\emph{quasiconformal}; the constant $K$ is called \emph{dilatation} of $f$.
The condition  $f\in W_{\rm loc}^{1,2}$ means that the first distributional partial derivatives of $f$ are locally in $L^2$. A quasiconformal map is necessarily orientation preserving.

If $(X, d_X)$ and $(Y,d_Y)$ are metric spaces, a homeomorphism $f\: X\to Y$  is called {\em quasisymmetric} or a \emph{quasisymmetry} if there exists a homeomorphism $\eta\:
[0,\infty)\to[0,\infty)$ such that 
%\begin{equation}\label{eq:defqs}
$$
\frac{d_Y(f(u),f(v))}{d_Y(f(u),f(w))}\leq\eta
\bigg(\frac{d_X(u,v)}{d_X(u,w)}\bigg),
%\end{equation}
$$
for every triple of distinct points $u,v,w\in X$. 
If we want to emphasize the dependence on the \emph{distortion function} $\eta$, we say that $f$ is $\eta$-\emph{quasisymmetric}, or $\eta$-\emph{quasisymmetry}.

Suppose  $U$ and $V$ are  subregions of $\hC$. Then 
every orientation-preserving $\eta$-quasisymmetric homeomorphism $f\: U\ra V$ is $K$-qua\-si\-con\-for\-mal with $K$ that depends only on $\eta$. Conversely, every properly normalized quasiconformal homeomorphism $f\: U\ra V$ is {\em locally quasisymmetric}, i.e., 
for every compact set $M\subset U$,  
the restriction $f|_M\: M\ra f(M)$ is a quasisymmetry, quantitatively, i.e., $\eta$ depends only on $K$ and the relative distance between $M$ and $\dee U$; see \cite[p.~58, Theorem 3.4.1 and p.~71, Theorem~3.6.2]{AIM} and~\cite[Theorem~11.14]{He}. This is referred to as the \emph{egg yolk principle}.  

If we have a family of maps with the same 
dilatation or distortion function, then we say that the family is {\em uniform}.
E.g., a family of homeomorphisms is \emph{uniformly quasisymmetric} if there exists a homeomorphism $\eta\:[0,\infty)\to[0,\infty)$ such that each map from the family  is $\eta$-quasisymmetric.
Inverses and compositions of quasiconformal or quasisymmetric maps are quantitatively quasiconformal or quasisymmetric, respectively.

According to the Ahlfors--Beurling theorem~\cite{BA}, each orientation preserving  $\eta$-quasisymmetric map $h\:\T\to\T$ has a $K$-quasiconformal extension $H$ to the whole complex plane, where $K$ depends only on $\eta$.
Conversely, the homeomorphic extension $h$ of each $K$-quasiconformal homeomorphism $H\:\D\to\D$ or $H\:\C\setminus\overline{\D}\to\C\setminus\overline{\D}$ to $\T$
is $\eta$-qua\-si\-sym\-met\-ric for some $\eta$ that depends only on $K$.  

Let $\xi$ be an orientation preserving homeomorphism of $\C$ that leaves $\Ju$ invariant. 
%Such a map must preserve bounded Fatou components since the unbounded Fatou component is not a Jordan domain. Hence it also preserves $U_\infty$.
Let $U, V$ be bounded Fatou components of $f$ such that  $\xi\: U\to V$. Then the homeomorphism $\xi_{U,V}=\phi_V\circ \xi\circ \psi_U\:\T\to\T$, where $\psi_U=\phi_U^{-1}$,  preserves the set of dyadic points.  Indeed, this follows from the fact that $\xi$ preserves the set of global cut-points of $\Ju$. 
By Lemma~\ref{homeos}, the global homeomorphism $\xi$ also induces a homeomorphism $\xi_\infty$ of $\T$ that satisfies  
$\psi_\infty\circ\xi_\infty=\xi\circ\psi_\infty$.
%This follows from the fact that $\xi$ preserves the cyclic order of limbs at each global cutpoint and induces a homeomorphism of prime ends for $U_\infty$.
For the same reason as above, the induced map $\xi_\infty$  preserves the set $D_\infty=\cup_{n=0}^\infty D_n$, where $D_n$ is the set of all preimages of $\{1/3,2/3\}$ under $g^n$, with $g(z)=z^2$. 
%The reason for using the inverse $\phi_\infty^{-1}$ rather than $\phi_\infty$ in the definition of $\xi_\infty$ above is that, unlike $\phi_\infty$, its inverse has a continuous extension to the boundary.

Let $\xi$ be a topologically extendable $\eta$-quasisymmetric self-map of $\Ju$. In what follows, we also denote by $\xi$ its homeomorphic extension to all of $\C$. Let $U$ and $V$ be bounded Fatou components of $f$ such that $\xi(U)=V$. In this case the maps $\xi_{U,V}$, where $U$ runs over all bounded Fatou components, are uniform quasisymmetries, i.e, they are $\eta'$-quasisymmetries with $\eta'$ that depends only on $\eta$.
This follows from an elementary fact that all the bounded Fatou components are uniform quasidiscs, which, in turn, is a consequence of the hyperbolicity of $f$.  
 In particular, each $\xi_{U,V}$ has a $K$-quasiconformal extension to $\C$, where $K$ depends only on $\eta$. The following lemma shows that the same holds for the map $\xi_\infty$.

\begin{lemma}\label{L:EmbQS}
Let $\xi$ be a topologically extendable $\eta$-quasisymmetric map of
$\Ju$. Let $\xi_\infty$ be the induced map on the unit circle $\T$
that satisfies the semi-conjugation 
$
\inv_\infty \circ\xi_\infty=\xi\circ\inv_\infty.
$
Then $\xi_\infty$ is an $\eta'$-quasisymmetric homeomorphism of $\T$ with $\eta'$ that depends only on $\eta$.  In particular, $\xi_\infty$ has a $K$-quasiconformal extension to $\C$, where $K$ depends only on $\eta$.
\end{lemma}
\no
\emph{Proof.}
It is enough to show that there exists $C>0$ that depends only on $\eta$, such that  if $I$ and $J$ are two adjacent non-overlapping arcs on $\T$ that have the same lengths, then 
$$
{\rm diam}(\xi_\infty(I))/C\leq{\rm diam}(\xi_\infty(J))\leq C {\rm diam}(\xi_\infty(I)).
$$

To prove this, we first show that there exist constants $C_1, C_1'>0$, such that if $I$ and $J$ are two adjacent non-overlapping arcs on $\T$, then 
$$
{\rm diam}(I)/C_1\leq{\rm diam}(J)\leq C_1{\rm diam}(I)
$$ 
if and only if
$$
{\rm diam}(I')/C_1'\leq{\rm diam}(J')\leq C_1'{\rm diam}(I'),
$$ 
where for an interval $I$ in $\T$ we denote by $I'$ the corresponding Julia arc, i.e., $I'=\psi_\infty(I)$; we use the same convention in the rest of this proof. 
In this statement, for the necessary part $C_1$ is given and $C_1'$ depends only on $C_1$, and for the sufficiency part it is the other way around. 

Indeed,
since $\psi_\infty$ does not collapse arcs of $\T$ to points, it follows that if $I$ is such an arc with ${\rm diam}(I)\geq \epsilon>0$, then there exists $\delta>0$ that depends only on $\epsilon$, such that 
$
{\rm diam}(I')\geq\delta.
$
Conversely, the uniform continuity of $\psi_\infty$ implies that if $\delta>0$ is given such that  
${\rm diam}(I')\geq\delta$, then there exists $\epsilon>0$ that depends only on $\delta$, with ${\rm diam}(I)\geq \epsilon$.

Now, if $
{\rm diam}(I)/C_1\leq{\rm diam}(J)\leq C_1{\rm diam}(I)
$, then  there exists a  constant $\epsilon_0>0$ that depends only on $C_1$, such that for some $n\in \N_0$ we have  
\begin{equation}\label{E:LargeArcs}
{\rm diam}(I_n), {\rm diam}(J_n)\geq\epsilon_0,
\end{equation}
where $I_n=g^n(I)$ and $J_n=g^n(J)$.
From the previous paragraph it follows that~\eqref{E:LargeArcs} holds if and only if there exists $\delta_0>0$ such that
$$
{\rm diam}(I_n'), {\rm diam}(J_n')\geq\delta_0.
$$ 
Here, for the ``if" part, $\epsilon_0$ depends only on $\delta_0$, and for the ``only if" part  $\delta_0$ depends only on $\epsilon_0$. Thus, there exist constants $C_2, C_2'>0$ that depend only on $\epsilon_0$ (or only on $\delta_0$), such that 
$$
{\rm diam}(I_n)/C_2\leq{\rm diam}(J_n)\leq C_2{\rm diam}(I_n)
$$ 
and
$$
{\rm diam}(I_n')/C_2'\leq{\rm diam}(J_n')\leq C_2'{\rm diam}(I_n').
$$ 
Note that $I_n'=f^n(I')$ and $J_n'=f^n(J')$.
We may and will assume that $n$ is chosen not too large, so that $I_n'$ and $J_n'$ are properly contained in an open set where the appropriate branch of $f^{-n}$ is well-defined and conformal.  The egg yolk principle now implies that there exists a constant $C_1'>0$ that depends only on $C_2, C_2'$, and hence only on $C_1$, such that 
$$
{\rm diam}(I')/C_1'\leq{\rm diam}(J')\leq C_1'{\rm diam}(I').
$$

The converse implication, namely that 
the last inequalities imply
$$
{\rm diam}(I)/C_1\leq{\rm diam}(J)\leq C_1{\rm diam}(I),
$$
for some $C_1>0$ that depends only on $C_1'$, follows the same lines with the egg yolk principle applied to $f^n$ rather than to $f^{-n}$. 

We are now ready to finish the proof. If $I$ and $J$ are two adjacent non-overlapping arcs of $\T$ that have the same lengths, then for some absolute constant $C_1'>0$ we have 
$$
{\rm diam}(I')/C_1'\leq{\rm diam}(J')\leq C_1'{\rm diam}(I').
$$ 
Since $\xi$ is $\eta$-quasisymmetric, there exists a constant $C_2'>0$ that depends only on $\eta$ and $C_1'$, such that 
$$
{\rm diam}(\xi(I'))/C_2'\leq{\rm diam}(\xi(J'))\leq C_2'{\rm diam}(\xi(I')).
$$
But, for each interval $I$ in $\T$, it follows from 
$\psi_\infty\circ\xi_\infty=\xi\circ\psi_\infty$ that $\xi(I')=(\xi_\infty(I))'$, and so
$$
{\rm diam}((\xi_\infty(I))')/C_2'\leq{\rm diam}((\xi_\infty(J))')\leq C_2'{\rm diam}((\xi_\infty(I))').
$$
Now we apply the above claim to conclude that there exists a constant $C>0$ that depends only on $C_2'$ with
$$
{\rm diam}(\xi_\infty(I))/C\leq{\rm diam}(\xi_\infty(J))\leq C{\rm diam}(\xi_\infty(I)),
$$
and the lemma follows.
\qed

\medskip

\section{Thompson group action on $\Ju$}\label{S:TG}

\subsubsection{Thompson groups} 
\no
The \emph{Thompson group} $F$ is a group of orientation preserving piecewise linear homeomorphisms of the closed interval $[0,1]$ whose break points, i.e., points of non-differentiability, are dyadic points and such that on intervals of linearity the slopes are integer powers of 2. See~\cite{CFP} for background on the group $F$ as well as the Thompson group $T$ of the unit circle, defined below. It follows immediately that the elements of $F$ preserve the set of dyadic points.
The  group $F$ is  generated by
$$
A(t)=\begin{cases}
t/2,\quad & 0\leq t\leq 1/2,\\
t-1/4,\quad & 1/2\leq t\leq 3/4,\\
2t-1,\quad & 3/4\leq t\leq 1,
\end{cases}
$$
and
$$
B(t)=\begin{cases}
t,\quad & 0\leq t\leq 1/2,\\
t/2+1/4,\quad & 1/2\leq t\leq 3/4,\\
t-1/8,\quad & 3/4\leq t\leq 7/8,\\
2t-1,\quad &7/8\leq t\leq 1.
\end{cases}
$$
%\comm{
\begin{figure}
[htbp]
\begin{center}
\includegraphics[height=60mm]{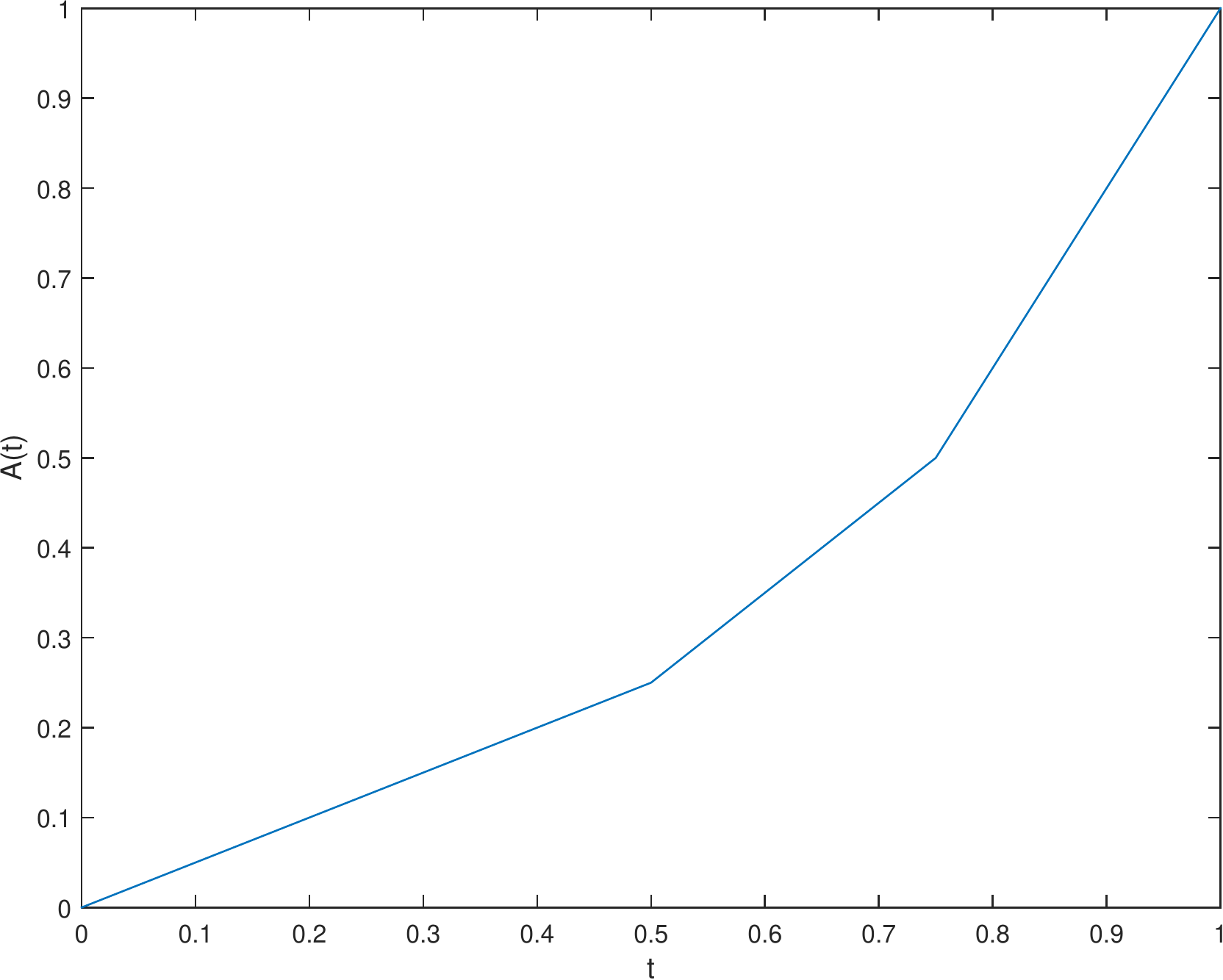}
\caption{
The graph of $A(t)$.
}
\label{F:A}
\end{center}
\end{figure}

\begin{figure}
[htbp]
\begin{center}
\includegraphics[height=60mm]{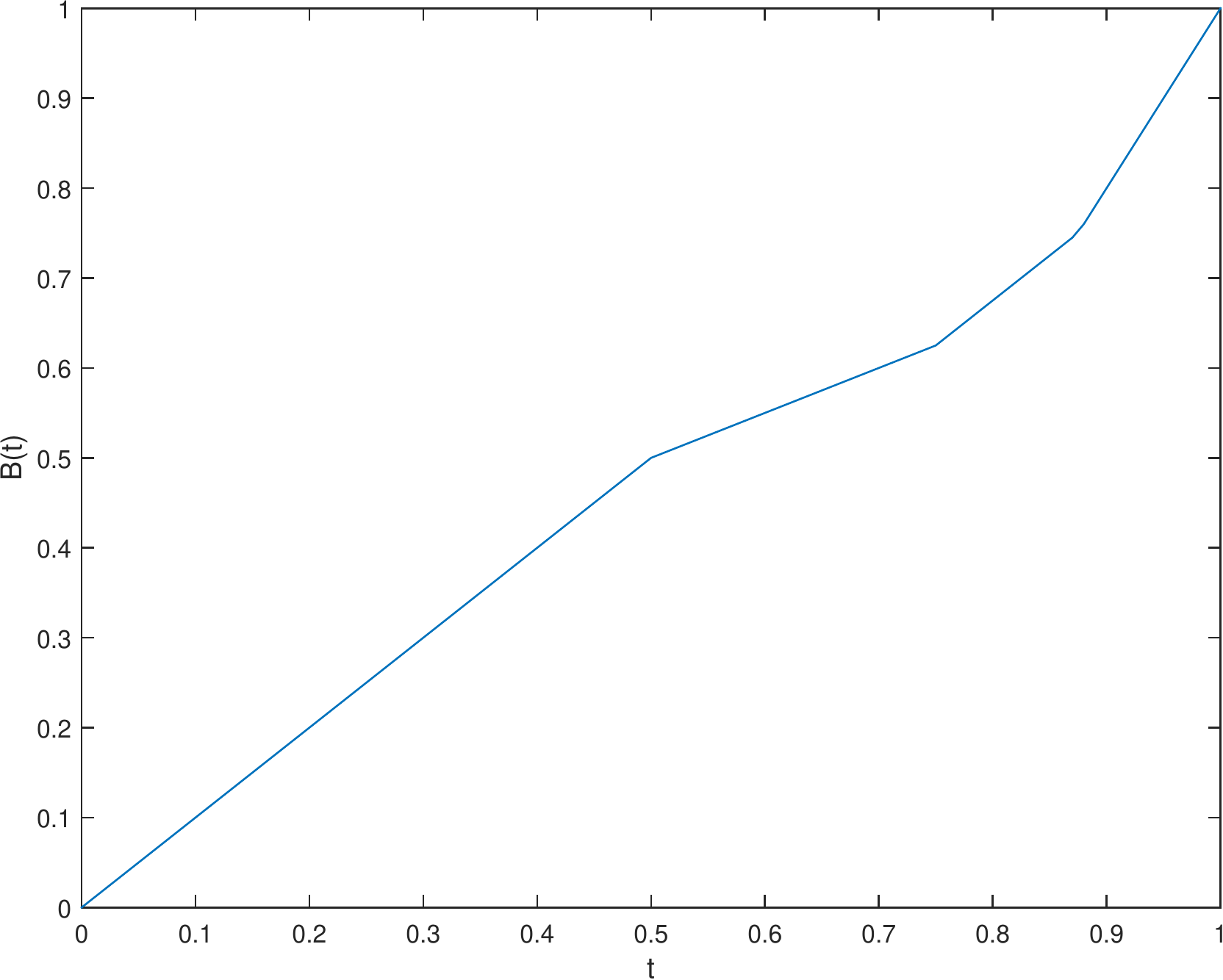}
\caption{
The graph of $B(t)$.
}
\label{F:B}
\end{center}
\end{figure}
%}

%\begin{figure}[htbp]
%\begin{minipage}[b]{0.5\textwidth}
%\centering
%\includegraphics[width=\textwidth]{Ac.pdf}
%%\caption{}
%\label{F:A}
%\end{minipage}
%\hfill
%\begin{minipage}[b]{0.5\textwidth}
%\centering
%\includegraphics[width=\textwidth]{Bc.pdf}
%%\caption{}
%\label{F:B}
%\end{minipage}
%\label{F:AB}
%\caption{The graphs of $A$ and $B$.}

Similarly, the \emph{Thompson group} $T$ is the group of orientation preserving piecewise linear (in the angular metric) homeomorphisms of the unit circle $\T$ that preserve the set of dyadic points, whose break points are dyadic points, and the slopes on intervals of linearity are integer powers of 2.
Elements of the Thompson group $F$ induce in the obvious way elements of $T$ that fix $1\in\T$. The group $T$ is generated by elements $A, B, C$, where $A$ and $B$ are the homeomorphisms of $\T$ induced by the elements $A$ and $B$ of $F$, respectively, and
$$
C(t)=\begin{cases}
t/2+3/4,\quad & 0\leq t\leq 1/2,\\
2t-1,\quad & 1/2\leq t\leq 3/4,\\
t-1/4,\quad & 3/4\leq t\leq 1.
\end{cases}
%\end{figure}
$$
%\comm{
\begin{figure}
[htbp]
\begin{center}
\includegraphics[height=60mm]{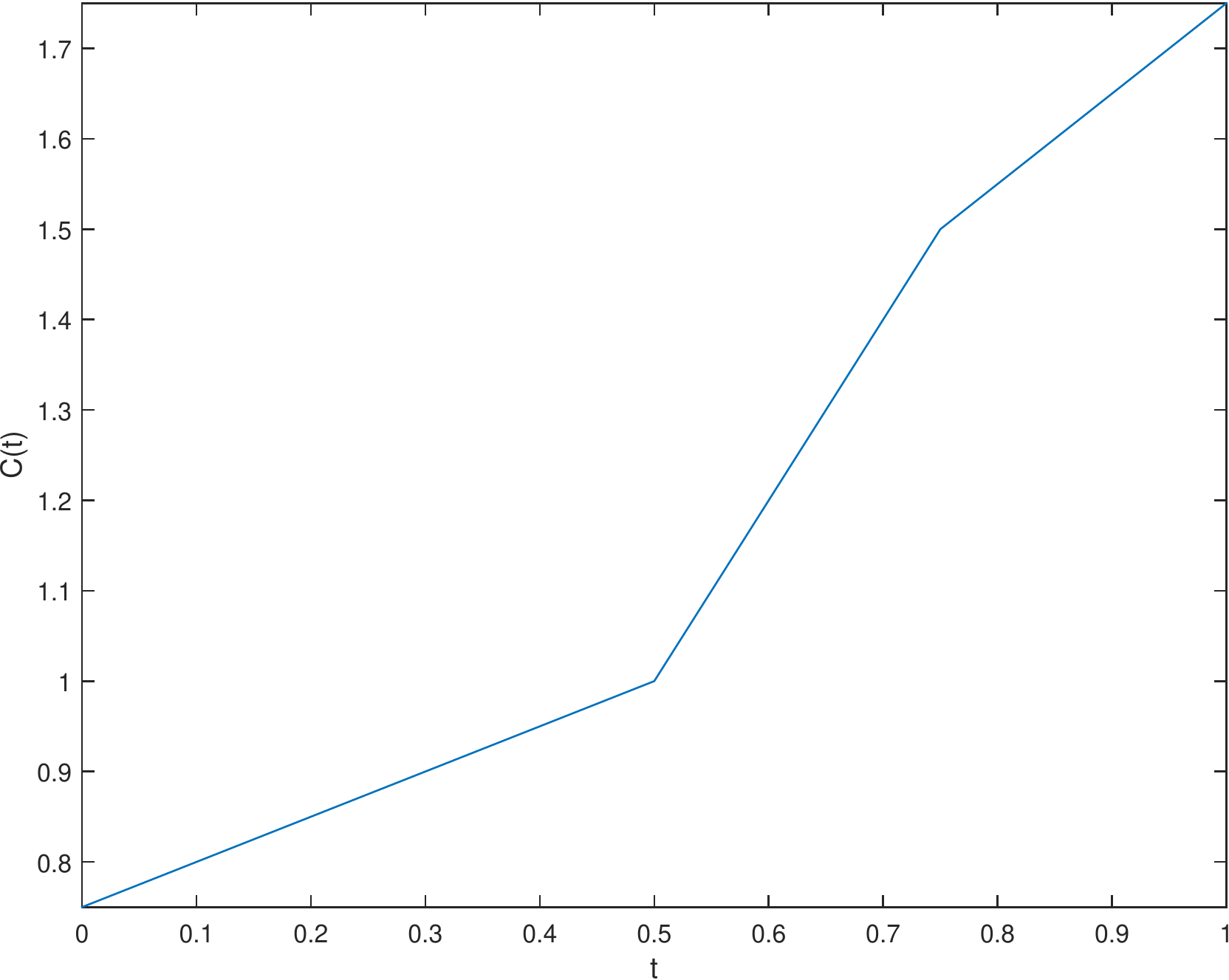}
\caption{
The graph of $C(t)$.
}
\label{F:C}
\end{center}
\end{figure}
%}
%}

\msk
\subsubsection{Pseudo-group $\Ups_g$} 
\no
A \emph{pseudo-group} $\Ups$ on a smooth manifold $M$ is a collection of diffeomorphisms $h\: U_h\to V_h$ between open subsets of $M$ that satisfy the following axioms:

\no
1. If $h_1$ and $h_2$ in $\Ups$ and $V_{h_1}\subseteq U_{h_2}$, then $h_2\circ h_1\in\Ups$; 

\no
2. If $h\in\Ups$, then $h^{-1}\in\Ups$;

\no
3. ${\rm id}_M\in\Ups$;

\no
4. If $h\in\Ups$ and $W\subseteq U_h$ is open, then the restriction $h|_{W}\in\Ups$;

\no
5. If $h\: U_h\to V_h$ is a diffeomorphism between open sets in $M$, and if for each $p\in U_h$ there is an open neighborhood $W$ of $p$ in $U_h$, such that $h|_W\in\Ups$, then $h\in\Ups$. 

For examples and background on pseudo-groups the reader may consult, e.g., \cite[Ch.~2]{CC00}.

For further reference, let us formulate two elementary lemmas:
%  will be useful in what follows.
\begin{lemma}\label{L:LargeScaleAppr}
Let $I_0$ and $I_1$ be two non-empty open arcs on $\T$ so that the endpoints are dyadic. 
Then there exists an orientation preserving piecewise linear map $\tau$ from $I_0$ onto $I_1$ whose break points  are dyadic points and such that on each interval of linearity, the slope of $\tau$ in the angular metric is $2^n$ for some $n\in\Z$.
\end{lemma}
\no
{\emph{Proof.}}
%This follows immediately from the fact that the Thompson group $T$ is 2-transitive. The 2-transitivity of $T$ is established as follows. 
First, we may assume by applying rotations by dyadic numbers that both $I_0$ an $I_1$ have a common endpoint at 1, and they can be identified with intervals in $[0,1]$ with one endpoint at 0. Then the other endpoints of $I_0$ and $I_1$ are dyadic points $d_0,d_1\in (0,1]$, respectively.
%Now we can essentially apply the fact that the Thompson group $F$ is transitive over dyadic points in $(0,1)$. 
%This can be seen as follows. 
We may assume that $d_0=k/2^n, d_1=m/2^n$ and $k<m$. We now apply the identity to the first $k-1$ intervals of length $1/2^n$ comprising $(0,d_0]$, and scale the last such interval by 2. All the slopes of such a map on $(0, d_0]$ are integer powers of 2 and this map takes $(0,d_0]$ to the interval $(0, d_0+1/2^n]$. Inductive process now finishes the proof. 
%; see, e.g., \cite[Corollary 2.\ (a)]{Sa}.
\qed

\medskip

Recall that $g$ stands for the doubling map of the circle. 
Given a path $\gamma\: [0,1]  \ra \T$, a moment $m\in \N_0$,
and a preimage $z_0\in g^{-m} (\gamma(0))$, 
we can uniquely lift the path $\gamma$ by $g^m $ to a path 
$\de\: [0,1]\ra \T$ with $\de (0) = z_0$.
Informally, we will refer to the corresponding analytic continuation
$\de(t) = g^{-m } (\gamma (t))$ of $g^{-m }$ along $\gamma$
as the \emph{branch} of  $g^{-m }$ on  $\gamma$  that \emph{starts} at $z_0$.

In particular, we can consider an arc $I\subset \T$,
a curve $g^n\: I \ra \T$,  and a branch of some $g^{-m}$
on this curve. We obtain  a composition
 $\xi= g^{-m}\circ g^n$  on $I$.
If the image $ \tl I: = \xi (I)$ is also an arc on $\T$
(i.e., $\xi\: I\ra \T$ does not ``wrap around the circle''),
then $\xi\: I\ra \tl I  $ is a linear (in the angular coordinate)
diffeomorphism with slope $2^{n-m}$.
Let us denote the pseudo-group of  such diffeomorphisms by $\Ups_g$.

\begin{lemma}\label{transitivity}
If $I$ and $\tilde I$ are arcs of $\T$ whose endpoints are dyadic, 
and if $\xi$ is a linear map of $I$ onto $\tilde I$ such that the slope of $\xi$ in the angular coordinate is an integer power of 2, 
then $\xi$ belongs to the pseudo-group $\Ups_g$. 
%is a composition of iterates of $g$ and appropriate branches of $g^{-1}$. 
%%% For any two dyadic intervals $I^n_i$ and $I^m_j$, 
%%%there exists a branch of $g^{-m}$ such that $g^{-m} \circ g^n$ maps
%%%linearly (in the angular coordinate)  $I^n_i$ onto $I^m_j$.
\end{lemma}
\no
\emph{Proof.}
Let $I=[a,b]$ and $\tilde I=[\tilde a,\tilde b]$.
Since $a$ is assumed to be dyadic, we can find $n \in\N_0$ such that $g^n (a) =1$.
Let us consider the branch of  $g^{-n}$ on the path $g^n\: I \ra \T$
that starts at $1$. The composition $r= g^{-n} \circ g^n\: I \ra \T$ 
is a linear (in the angular coordinate)  map with slope $1$.
Hence it rotates $I$ to a dyadic interval $J= [1, e(\theta)]$, $\theta\in [0, 1]$. 

Likewise, we can find a moment $\tilde n\in \N_0$ and a branch of $g^{- \tilde n}$
such that the composition $\tilde r = {g}^{-\tilde n} \circ {g}^{\tilde n}\: \tilde I \ra \T$ 
rotates $\tilde I$ to a dyadic interval $\tilde J = [1, e(\tilde \theta)]$, $\tilde \theta\in [0, 1]$.
Moreover, $\tilde \theta = 2^k \theta$, where $2^k$ is the slope of $\xi$ on $I$, $k\in \Z$. 
Hence $\tilde J = g^k (J)$, 
where, in the case when $k$ is negative, 
the branch of $g^k$ is chosen so that the point 1 is fixed.

%Since the endpoints $a$ of $I$ and $\tilde a$ of $\tilde I$ are assumed to be dyadic, 
%we can find $n, \tilde n\in\N_0$ that have the following properties. 
%The lift under $g^n$ along the curve $\{g^n(t)\: t\in I\}$ satisfies $g^{-n}(g^n(a))=1$. 
%Likewise, the lift under $g^{\tilde n}$ along the curve 
%$\{g^{\tilde n}(t)\: t\in \tilde I\}$ satisfies $g^{-\tilde n}(g^{\tilde n}(\tilde a))=1$.
%Therefore, the maps $r=g^{-n}\circ g^n$ and $\tilde r=g^{-\tilde n}\circ g^{\tilde n}$  
%are rotations that take respectively $a$ and  $\tilde a$ to 1.

% Now,  % if the slope of $\xi$ on $I$ is $2^k,\ k\in\Z$, 
The desired map can now be written as  $\tilde r^{-1}\circ g^k\circ r$.  
% where, in the case when $k$ is negative, 
% the branch of $g^k$ is chosen so that the point 1 is fixed. 
%The map $g^n$ linearly maps $\inter I^n_i$ onto $\T\sm \{ 1\} $,
%and there is an inverse branch of $g^{-m}$ that linearly maps   
% $\T\sm \{ 1\} $ onto $\inter I^m_j$. The conclusion follows. 
\qed

\medskip

%Recall that the second iterate $h=f^2$ of $f$ preserves $U_0$ and it is conjugate via the B\"ottcher coordinate $\phi_{0}$ to $g(z)=z^2$. In the angular metric on $\T$ the map $g$ is linear with derivative 2. Thus, any inverse branch of $g$ is also linear with derivative $1/2$. 
%Therefore the action of the group $T$ on the boundary $\dee U_0$ can be described  using only the iterates of the map $h$ and appropriate branches of $h^{-1}$. E.g., the generator that corresponds to $A$, still denoted by $A$, can be obtained as follows. 
%On the lower arc of $\dee U_0$ between the fixed point $\alpha$ and $-\alpha$, the map $A$ is $h^{-1}$, where the inverse branch is chosen so that $h^{-1}(\alpha)=\alpha$. 
%On the arc of $\dee U_0$ between $-\alpha$ and the point $t_0$ of intersection of $\dee U_0$ with the positive imaginary axis, the map $A$  is $h^{-2}\circ h^2$, where the branch of $h^{-2}$ is such that $h^{-2}(\alpha-)=-t_0$. Here $\alpha-$ denotes the endpoint of the arc $\dee U_0\setminus\{\alpha\}$, as approached from the lower half-plane. Finally, 
%on the remaining arc of $\dee U_0$ between $t_0$ an $\alpha$, the map $A$ is $h$. The generators that correspond to $B$ and $C$, still denoted by $B$ and $C$, respectively, can be described similarly. In what follows we do not distinguish between elements of $T$ and the corresponding self-maps of $\dee U_0$ in the group generated by $A, B$, and $C$.

Let us say that a  piecewise linear homeomorphism
 $\T\ra \T$ is  {\em piecewise dynamical}
if its restriction to  every interval of linearity belongs to the
pseudo-group $\Ups_g$.  

\begin{corollary}\label{Thompson and dynamics} 
  Any element of the Thompson group $T$ is piecewise dynamical. 
%a piecewise composition of iterates of 
%the doubling map $g$ and its inverse branches. 
\end{corollary}

\medskip

For instance, the $A$-generator of the Thompson group can be obtained as follows:
$$
A=\begin{cases}
g^{-1},\quad & 0\leq t\leq 1/2,  \quad {\mathrm{where\ the\ branch}}\ 
                          g^{-1} \ {\mathrm{fixes\ 1}},   \\
g^{-2}\circ g^2,\quad & 1/2\leq t\leq 3/4, \quad   
  {\mathrm{where}}\  g^{-2}\:   (0,1) \ra (1/4, 1/2),  \\
g,\quad & 3/4\leq t\leq 1. 
\end{cases}
$$

\msk
\subsubsection{Piecewise dynamical action of the Thompson group on $\Ju$}
\no
Similarly to $\Ups_g$, let us consider the \emph{dynamical pseudo-group} $\Upsilon_f$
comprising local isomorphisms $\psi=  f^{-m} \circ f^n\: V \ra W$,
where $m,n\in\N_0$, $V$ is simply connected, 
 $f^n|_V$ is univalent,
%  does not have postcritical points, 
 and  $f^{-m}\circ f^n|_V $ is a well defined analytic branch of the 
algebraic function $f^{-m}\circ f^n$.   Note that
 in the log-B\"ottcher coordinate $u= \log \phi_\infty (z) $,
the map $\psi$ is affine with slope $2^{n-m}$, 
$u\mapsto 2^{n-m} u + c$, $c\in \R$.  
%Also, note:
%
%\begin{remark}\label{uniqueness} 
%Assume we have two elements of $\Upsilon_f$ defined on $V$,
%$\psi=  f^{-m} \circ f^n$ and $\tilde \psi=  f^{-\tilde m} \circ f^{\tilde n}$.
%Let  $z_0\in V$ be such that the points 
%$f^k z_0$, $k < \max (n, \tilde n)$, are not periodic.  
%If   $\psi(z_0)= \tilde \psi (z_0)$, then $\psi \equiv \tilde \psi$.
%% is uniquely determined by the value $\phi(z_0)$;
%\end{remark}

Let us say that a homeomorphism $\xi\: J\ra \tilde J$ between Julia arcs  is {\em piecewise dynamical}
if $J$ can be decomposed into finitely many Julia arcs $J_i$ that share only global 
cut-points, so that each restriction $\xi|_{J_i}$,  extended to some
neighborhood $W_i$,  belongs to the pseudo-group $\Ups_f$. If $J=\tilde J=\Ju$, piecewise dynamical homeomorphisms $\xi\: \Ju\ra \Ju$ form a group, and we denote it by $\Dyn (\Ju)$. 
%be the group of piecewise dynamical homeomorphisms  $\xi: \Ju\ra \Ju$.

\comm{*****
Let $z_1$ and $z_2$ be (possibly the same) global cut-points of $\Ju$, and let $J$ be a connected component of $\Ju\setminus\{z_1,z_2\}$. A \emph{piecewise dynamical} map $\xi$ defined on $J$ is a homeomorphism of $J$ onto a connected component of $\Ju\setminus\{\tilde z_1,\tilde z_2\}$, where $\tilde z_1, \tilde z_2$ is another pair of global cut-points of $\Ju$, such that $J$ is a union of finitely many components on each of which $\xi$ is a composition of integer iterates of $f$. We also say that a global homeomorphism $\xi$ of $\Ju$ is \emph{piecewise dynamical} if it is piecewise dynamical when restricted to each $J$, a connected component of $\Ju\setminus\{z_1,z_2\}$ for some global cut-points $z_1, z_2$ of $\Ju$. 
**}

Let us also say that a map $h\: U\ra V$ between two bounded Fatou components 
% of $\inter \KK$
\emph{respects the B\"ottcher coordinate} if it is the identity  
in the B\"ottcher coordinates of these components, 
i.e., 
$$
    \phi_V \circ h \circ \psi_U = \id\: \D \ra \D.
$$

\begin{lemma}\label{lem:thom}
The Thompson group $T$ induces the group of
piecewise dynamical  quasisymmetries of $\Ju$ 
that keep the central Fatou component $\dee U_0$ invariant. 
Moreover, they admit global quasiconformal extensions to $\C$ 
that respect the B\"ottcher coordinates
of all non-central bounded Fatou components. 
\end{lemma}
\no
\emph{Proof.}
Let $\xi_0$ be an element of $T$ acting on $\dee U_0$
piecewise linearly in the (inner) angular coordinate.
We want to extend it to 
an orientation preserving   quasisymmetric    % quasisymmetric 
homeomorphism of $\Ju$
(actually, to a global quasiconformal map of $(\C, \Ju)$). 

% Note first that since $\dee U_0$ is   $\xi_0$ is qs  extends quasiconformally

Let $I\subseteq \dee U_0$ be an arc of linearity for $\xi_0$.
By Corollary \ref{Thompson and dynamics},
$\xi_0|_I =  f^{-m}\circ f^n$  for some $n,m\in \N_0$.
Hence $\xi_0$ admits an analytic extension $\hat \xi_0$ to a neighborhood $V$ of $I$ as an element of the dynamical pseudo-group $\Upsilon_f$.
%Select this neighborhood so that it is simply connected and has simply connected intersections with each wake $S_z$ attached to $I$.
%Moreover, in the log-B\"ottcher coordinate of  $V\cap U_\infty$, $\xi_0$ is an affine map 
%with slope $2^{n-m}$. 

Let us puncture out from $I$ all dyadic points of level $\leq n$. 
Take any complementary (open)  interval $J \subseteq I$. 
Let $\Omega_J$ be the region bounded by $J$ and two external rays landing at 
$\dee J$ (where the rays are selected so that $\Omega_J$ does not
intersect the wakes attached to $\dee J$). 
Since $ f^n(\Omega_J) $ is disjoint from $U_0\cup U_{-1}$, 
all branches of $f^{-m}$ are well defined on it, so $\hat\xi_0$ extends
analytically to $\Omega_J$ (mapping it to some  $\Omega_{\tilde J}$).  We restrict this map to the intersection of $\Omega_J$ with the filled Julia set $\KK$ and extend it continuously to the closure $\KK_J$ of this set.
%Call this extension $\xi_J\in \Upsilon_f$. 
%%%(In fact, it  extends analytically to a  neighborhood of $\bar
%%%\Pi_J$, so it induces a qs map of $\KK_J\ra \KK_{\tl J}$,
%%%%where $\KK_J= \KK\cap \bar\Pi_J$.) 

In this way, we extend $\xi_0$ to all $\KK\sm U_0$,
except finitely many limbs %wakes 
attached to $\dee U_0$.
Let $S_z$ be one of the wakes containing such a limb.
% and let $z\in \dee U$ be its root; 
Its root  $z\in \dee U_0$ is a dyadic point of some level $k$.
Let $\tl z= \xi_0(z)$. Since $\xi_0$ is Thompson,
$\tl z$ is also dyadic (of some level $\tl k$), 
so there is  a wake  $ S_{\tl z} $ rooted at $\tl z$.

Recall that  $S_0$ is the sector bounded by $\rho^{\pm\frac13}$ that contains $U_{0}$.
Then $f^k$ and $f^{\tl k}$ conformally  map 
% $S$ and $\tl S$
 $S_z$ and $S_{\tl z} $, respectively, onto $S_0$. (If $k$ or $\tilde k$ is 0, then $f$ takes $S_z$, respectively $S_{\tl z}$, to $S_0$.)

Thus, we obtain a map $\xi_z  = f^{-\tl k}\circ   f^k\: S\ra  \tl S $
of pseudo-group $\Ups_f$. 
Restricting it to the filled Julia set,
%  Letting $\KK_z= S_z\cap \KK$,
we obtain a homeomorphism $\xi_z\: \KK_z \ra \KK_{\tl z}$
between the corresponding limbs. 
 Putting these maps together,
we obtain an extension of $\xi_0$ to % all the  wakes of interest.
% Altogether, this provides us with 
a homeomorphism 
$\xi\: \KK \sm U_0 \ra \KK\sm U_0$ that restricts to an element in $\Dyn(\Ju)$.   
%Finally, let us extend $\xi_0$ to finitely many ``punctured out''  wakes 
% as above. 
%Let $S_z$ be a wake attached at some point $z\in J$. 
%Above we have constructed  a conformal  map $\psi_z: S_z \ra S_{\xi_0 (z)} $ which also belongs to the 
%pseudo-group $\Upsilon _f$ and whose boundary value at $z$ is equal to $\xi_0(z)$. 
%By Lemma \ref{uniqueness},  $\psi_z$ coincides with $\tilde \xi_0$ on $V\cap S_z$.
%Thus, all the maps $\psi_z$ (with $z\in J$) glue into a single conformal map on
%a neighborhood $V\cup \bigcup S_z$ of $\KK_J$. 
%(The latter is part of the filled Julia set attached to $J$).

Let us show that the map $\xi$ admits a global quasiconformal extension to the complex
plane. By definition, it has a conformal extension to each non-central bounded Fatou component as a conformal map $f^{-m}\circ f^n$.  Since $ \dee U_0 $ is a quasicircle and 
$\xi_0\: \dee U_0 \ra \dee U_0 $ is piecewise linear in
the inner angular coordinate, it is quasisymmetric. 
Hence it admits a quasiconformal extension to $U_0$
(Ahlfors-Beurling extension \cite{BA}). 

We now extend $\xi$  to $U_\infty$.
Let $J_i$ be the arcs of $\dee U_0$
 considered above, and let $z_i$  be their boundary  dyadic points.
Let $\YY_i = \dee \KK_{J_i}$ and $\ZZ_i = \dee \KK_{z_i}$ be the corresponding Julia arcs.   
In  the B\"ottcher coordinate, 
these  Julia arcs  correspond to some arcs $Y_i$ and $Z_i$   
tessellating  the circle $\T$. 

Since $\xi$ belongs to the dynamical pseudo-group $\Ups_f$ 
on each $\YY_i$ and $\ZZ_i$,
%  $ \phi_\inft\circ \xi\circ   \phi_\infty^{-1} : U_\infty\ra \U\infty$
%  admits  an extension to  
it induces linear 
(in the outer angular coordinate of the circle $\T$)  
 maps  $Y_i\ra \tl Y_i$ and $Z_i\ra \tl Z_i$ .
Moreover, since $\xi_0$ preserves the
cyclic order in which  the arcs $J_i$ and points $z_i$ appear on $\dee U_0$, 
these maps form  a single piecewise linear
homeomorphism $h\: \T\ra \T$. Applying the Ahlfors-Beurling extension once
again, we obtain a quasiconformal extension  $\hat h$ of $h$ to $U_\infty$.

Note that $h$ preserves the basilica lamination on $\T$. 
It is so on 
$$
 \T\sm \bigcup \di Z_i \quad  (=\T \sm \bigcup \di Y_i)
$$   
since $\xi$ 
admits local homeomorphic extensions to $\C$ near  any point of  $\Ju\sm \{z_i\}$. 
It is also true at the points of $\bigcup \di Z_i$.
Indeed, the lamination pairs the boundary points of each
$Z_i$ (corresponding to the rays enclosing the limb $\ZZ_i$). 
Since $\xi$ maps the  limb $\ZZ_i$ to another limb $\tl \ZZ_i$, 
the boundary points of $\tl Z_i=h(Z_i)$ 
correspond to the rays enclosing $\tl \ZZ_i $. So, they are paired
by the lamination. 
% came from a homeomorphism  ($\xi_0)$) of $\di U_0$.
Hence $\hat h$ descends to a homeomorphism of $(\C, \KK)$ 
%Going back to the original dynamical plane, we obtain 
providing us with the desired quasiconformal extension of $\xi$ to $\C$. Here we also need the fact that the Julia set $\Ju$ is removable for quasiconformal maps~\cite[Corollary~2 on
p.~5 and Remark on p.~3]{Jo}. We continue to denote this extension by $\xi$.

It remains to show that $\xi$ respects the B\"ottcher coordinates of all non-central bounded Fatou components. In each such component $U$ it has the form $f^{-m}\circ f^n$ and maps it to another such component $\tilde U$. The claim now follows from the fact that the B\"ottcher coordinate of a bounded Fatou component with dynamical distance $k$ to $U_0$ is $\phi_0\circ f^k$. Indeed, since in the construction  above $n$ is chosen so that $f^n(U)$ does not contain the postcritical points, we have that $n$ is at most the dynamical distance from $U$ to $U_0$. If $l\in\N_0$ is such that $n+l$ is the dynamical distance from $U$ to $U_0$, then for the branch of $f^{-l}$ that takes $U_0$ to $f^n(U)$, the map $\xi$ in $U$ is given by
$
f^{-m}\circ f^{-l}\circ f^{n+l}.
$
We conclude that $m+l$ is the dynamical distance from $\tilde U$ to $U_0$ and for an appropriate branch of $f^{-(m+l)}$ the map $\xi$ in $U$ is given by
$$
f^{-(m+l)}\circ f^{n+l}.
$$
This is equivalent to saying that $\xi$  respects the B\"ottcher coordinates.
\qed

\medskip

\subsubsection{Thompson-like action}
\no
The  Thompson group $T$  is based on dy\-adic points, 
iterated preimages of the fixed point $1$ under
the doubling map $g$. More generally, one can consider iterated
preimages of any periodic cycle 
$\bal$ and define the associated  \emph{Thompson-like} group $T_\bal$. I.e., the elements of $T_\bal$ preserve the set $D_\bal$ of all preimages of the elements of $\bal$ under the iterates of the map $g$,  the break points are at $D_\bal$, and the slopes are integer powers of 2. 
% (In what follows we only need to consider the case $\bal=\{1/3, 2/3\}$.)
The above discussion readily extends to this setting,
and in particular, we have the following counterpart of Corollary~\ref{Thompson and dynamics}:

\begin{corollary}\label{Thompson and dynamics 2} 
  Any element of the Thompson-like group $T_\bal$ is piecewise dynamical. 
%a piecewise composition of iterates of 
%the doubling map $g$ and its inverse branches. 
\end{corollary}

In what follows, we let $\bal=\{1/3, 2/3\}$. 
Let $T_\bal^b$ be the subgroup of the Thompson-like group $T_\bal$ consisting of homeomorphisms
$\T\ra \T$ that preserve the basilica lamination.

\begin{lemma}\label{L:homeos2}
Any  piecewise dynamical homeomorphism  $\xi\: \Ju\ra \Ju$ induces a  homeomorphism $\xi_\infty\: \T \ra \T$ that belongs to $T_\bal^b$.
% the Thompson-like group 
% $T_\bal$ with $\bal= \{ 1/3, 2/3\}$ that  preserves the basilica lamination.   
Moreover, the map $\Dyn (\Ju) \ra T_\bal^b$,   $\xi\mapsto\xi_\infty$, 
is a group isomorphism. 
\end{lemma}
%
%\begin{lemma}\label{L:homeos2}
%A map $\xi$ of $\Ju$ is piecewise dynamical if and only if the corresponding map $\xi_\infty$ of $\T$ satisfies:
%
%\smallskip\noindent
%-- it preserves the basilica lamination,
%
%\smallskip\noindent
%-- it is piecewise linear in the angular metric, 
%
%\smallskip\noindent
%-- it has breakpoints at the ends of the basilica lamination,
%
%\smallskip\noindent
%-- on each interval of linearity its slope is an integer power of 2.
%
%\smallskip\noindent
%Moreover, 
%
%\smallskip\noindent
%-- the map $\xi\mapsto\xi_\infty$ is an isomorphism of  groups under the composition.
%\end{lemma}
\noindent
{\emph{Proof.}}
It follows from the definition that any piecewise dynamical homeomorphism $\xi$ of $\Ju$ has a homeomorphic extension to $\overline{U_\infty}$, and therefore, by Lemma~\ref{homeos}, the induced  homeomorphism $\xi_\infty$ of $\T$ preserves the basilica lamination. The rest of the properties of $\xi_\infty$ follow from the assumption that $\xi$ is piecewise dynamical.
%Moreover, the homeomorphism $\xi_\infty$ is necessarily piecewise linear in the angular metric with breakpoints at the endpoints of the basilica lamination and with the slopes that are integer powers of 2. 

Conversely, any such map $\xi_\infty$ extends to a homeomorphism of $\C\setminus\D$ and hence descends  to a homeomorphism $\xi$ of $\Ju$ by Lemma~\ref{homeos}. The homeomorphism $\xi$ has to be piecewise dynamical because such is $\xi_\infty$ according to Corollary~\ref{Thompson and dynamics 2}.
%on each interval $I$ of linearity  the map $\xi_\infty$ has slope that is an integer power of 2, and hence by Lemma~\ref{transitivity} it can be represented as a composition of iterates  of $g$ and $g^{-1}$. 
%Indeed, an endpoint $a$ of $I$ is a dyadic point that is mapped by $\xi_\infty$ to a dyadic point $\tilde a$. 
%Therefore, we can find $n, \tilde n\in\N_0$ such that for appropriate branches of $g^{-n}, g^{-\tilde n}$, the maps $r=g^{-n}\circ g^n$ and $\tilde r=g^{-\tilde n}\circ g^{\tilde n}$  are rotations that take respectively $a$ and  $\tilde a$ to 1. Now, if the slope of $\xi_\infty$ on $I$ is $2^k,\ k\in\Z$, the desired map is $\tilde r^{-1}\circ g^k\circ r$, where, in the case when $k$ is negative, the branch of $g^k$ is chosen so that the point 1 is fixed. 
%

The statement about the map $\xi\mapsto\xi_\infty$ being an isomorphism follows immediately from the relation 
$$
\inv_\infty \circ\xi_\infty=\xi\circ\inv_\infty .
$$
%and the fact that for each $\xi$ the corresponding map $\xi_\infty$ is unique.
\qed

\medskip

One element of $T$ that will be useful in what follows is the rotation $\rho(z)=-z$. It is given by $f^{-2}\circ f^2$, where the branch of $f^{-2}$ is chosen so that $f^{-2}(\alpha)=-\alpha$.

\section{Inversion of $\Ju$}\label{S:Inv}
\no
In this section we define an inversion $\iota$ of $\Ju$ that along with the action of the Thompson group $T$ on $\Ju$  generates a group whose elements approximate every quasisymmetric self-map of $\Ju$ quantitatively. We recall that $\Ju_\alpha=\dee\KK_\alpha$, where $\KK_\alpha$ is the limb rooted at $\alpha$ that contains the boundary $\dee U_{-1}$ of  the bounded Fatou component $U_{-1}$ containing $-1$. Also, $\Ju_0=(\Ju\setminus\Ju_\alpha)\cup\{\alpha\}$.
 
\begin{lemma}\label{lem:inv}
There exists a quasisymmetric homeomorphism $\iota$ of $\Ju$ that fixes $\alpha$, 
interchanges $\Ju_\alpha$ and $\Ju_0$, and satisfies
$$
\iota^2=\id.
$$
Moreover, $\iota$ has a  global quasiconformal extension to $\C$
 and respects the B\"ottcher coordinates of all  bounded Fatou components.
\end{lemma}
\noindent
\emph{Proof.}
We define $\iota$ by $\iota=f$ on  $\Ju_\alpha$, and $\iota=f^{-1}$ on $\Ju_0$. Here we choose the branch of $f^{-1}$ that fixes $\alpha$.  We prove that so defined $\iota$ is quasisymmetric by showing that it extends to a quasiconformal map of $\C$. 

In each bounded component of the Fatou set whose closure intersects $\Ju_\alpha$ we extend $\iota$ by $\iota=f$, and in each such component whose closure intersects $\Ju_0$  we extend it by $\iota=f^{-1}$. In particular, it is immediate that if $U$ is any bounded Fatou component and $\dee \tilde U=\iota(\dee U)$, then this extension of $\iota$  respects the B\"ottcher coordinates of $U$ and $\tilde U$. 

The map $\iota$ has obvious conformal extensions to both wakes in $\C$ bounded by the two external rays in $U_\infty$ landing  at $\alpha$.
These extensions are conjugate by $\phi_{\infty}$ to the maps $g(z)=z^2$ for $1/3\le e(\theta)\le 2/3$ and $z\mapsto -\sqrt{z}$ for $-1/3\le e(\theta)\le1/3$, respectively, where the principal branch of the square root is selected. 
We can use the Ahlfors--Beurling extension~\cite{BA} to extend this piecewise linear map to a quasiconformal map of $\C\setminus{\overline\D}$ onto itself. This quasiconformal extension is conjugated back by $\phi_\infty$ to a quasiconformal map of $U_\infty$ onto itself that agrees with $\iota$ on the boundary of $U_\infty$. 

Combining the above quasiconformal extension of $\iota$ with the conformal extensions into bounded Fatou components, we obtain a homeomorphism of $\C$ that is quasiconformal on the Fatou set. But the Julia set of a postcritically finite polynomial is removable for quasiconformal maps~\cite{Jo}, and so we get a quasiconformal map of $\C$. Such a map is quasisymmetric, and hence $\iota$ is quasisymmetric on $\Ju$.  
\qed

\medskip

We call the homeomorphism $\iota$ the \emph{inversion} with respect to $\alpha$. 
The composition $\sigma = \rho\circ\iota$ with the rotation $\rho$, given by $\rho(z)=-z$,  
acts as a \emph{shift} to the right by one on the infinite chain $\mathcal C$ of bounded Fatou components 
that intersect the real line. This follows from the fact that  $\sigma$ preserves $\mathcal C$ 
and takes $\dee U_{-1}$ to $\dee U_0$. 
Moreover, $\sigma$ respects the B\"ottcher coordinates of all bounded Fatou components, 
except that $\sigma\: U_{-1} \ra U_0$ is $ 180^\circ$-rotation in the B\"ottcher coordinates. 
(In general, we will say that  a map $h\: U\ra V$ between bounded Fatou components 
\emph{rotates} the B\"ottcher coordinate if
$\phi_U\circ h\circ \psi_V$ is a rotation of the disk $\D$.)  
  
These remarks allow us to derive the following property:

\begin{lemma}\label{L:Trans}
The group $\hat T$ generated by the Thompson group $T$ and $\iota$ acts transitively on the vertices of $\TT$.
Moreover, let $U$ be an arbitrary bounded Fatou component,
and let $p$ be the shortest combinatorial chain in $\TT$ that joins $U_0$ to $U$. Then
there exists $t\in\hat T$ 
% such that $t(U)=U_0$ and $t$ has 
with the following properties:

\smallskip\noindent 
{\rm (i)} $t(U)=U_0$ and $t$ respects the B\"ottcher coordinate;
%  the identity map in the B\"ottcher coordinates of $U$ and $U_0$;

\smallskip\noindent 
{\rm (ii)} for any bounded Fatou component $V\notin p$, the map $t\:  V\ra t(V)$ 
respects the B\"ottcher coordinate;
%%% has an extension as
%  is the identity in the B\"ottcher coordinates of $V$ and $t(V)$; 

\smallskip\noindent
{\rm (iii)} for any bounded Fatou component $V\in p$,  the map $t\:  V\ra t(V)$
rotates the B\"ottcher coordinate by some dyadic angle.
% $t$ 
%%%% boundary map of $t$ 
%  is a rotation in the B\"ottcher coordinates of $V$ and  $t(V)$. 

% \smallskip\noindent
% -- finally,  we may assume that $t$ is the identity map in the B\"ottcher coordinates as the map from $U$ to $U_0$.
\end{lemma}

\noindent
{\emph{Proof.}}
% It is enough to show that given any bounded Fatou component $U$ 
%one can find an element $t$ of the group $\hat T$ such that $t(U)= U_0$ and $t$ has the properties described in the statement. 
%
Suppose that $p=\{U_0, U_1,\dots, U_k=U\}$. %  is the shortest combinatorial chain that connects $U_0$ to $U$. 
We will argue by induction on $k$. The base $k=0$ is obvious.
In general,  
the  component $U_1$ touches $U_0$, and so the base point of $U_1$ corresponds to a dyadic point on $\dee U_0$. The group $T$ 
contains rotations by any dyadic angle, and hence 
acts transitively on the dyadic points of $\dee U_0$. Thus there exists an element $t_1\in T$ such that $t_1(U_1)= U_{-1}$.  The elements of $T$ preserve the combinatorial distance on $\TT$, and therefore $t_1(U)$ has the same combinatorial distance $k$ to $U_0$ as $U$. Now we apply the shift $\sigma$  % $\rho\circ\iota$ 
to take  $U_{-1}$    % $t_1(U_1)$ 
to $U_0$. 
The combinatorial distance from $\sigma \circ t_1(U)$ to $U_0$ is $k-1$, and we can apply the inductive hypothesis.
Let $t_2\in \hat T$ be an element given by the inductive hypothesis that takes $\sigma\circ t_1(U)$ to $U_0$.

The map $t'= t_2\circ\sigma\circ t_1\in \hat T $ takes $U$ to $U_0$. 
Moreover, $\sigma\circ t_1$ respects the B\"ottcher coordinates of all non-central  bounded Fatou components
except $U_1$, on which it rotates the coordinate by $180^\circ$.
On $U_0$ itself, it rotates the B\"ottcher 
coordinate by some dyadic angle. 
 By the inductive assumption,
$t' \in \hat T $ respects the B\"ottcher coordinate for all $V\notin p$
and rotates it for all $V\in p$ by some dyadic angles. 
%
%The desired properties of $t$ follow from the above inductive construction and
%%% as well as the definition of 
%the properties of the actions of the Thompson group $T$ and the shift $\sigma$ on % $\Ju$.
%the bounded Fatou components.
%
% Indeed, the elements of $T$ act on all bounded Fatou components, 
% other than $U_0$, as the identity in the corresponding B\"ottcher coordinates. 
% Also, the inversion $\iota$  acts as the identity in the  B\"ottcher coordinates of all bounded Fatou components. 
Finally, % since $t(U)=U_0$, 
by postcomposing $t'$ with a rotation $\gamma$  in $T$, 
we can ensure   %may assume 
that  $t= \gamma\circ t'\: U\ra U_0$ respects the B\"ottcher coordinate as well.
%it is the identity in the B\"ottcher coordinates of $U$ and  $U_0$.
\qed

\medskip

\section{Proof of Theorem~\ref{thm:thom}}\label{S:Proof}
\no
Let $n\in N$ be arbitrary and let $\Pi_n=\{J_k,\ k=1,2,\dots, 2^{n+1}\}$ be the partition of $\Ju$ by Julia arcs described in Section~\ref{S:Lam}. We replace the map $\xi$ restricted to each $J_k$ by a piecewise dynamical map as follows.

\begin{lemma}\label{L:ConfElev}
Let $\xi$ be a topologically extendable $\eta$-quasisymmetric map of $\Ju$. Then there exist $N\in \N_0$ and a finite family $\mathcal F$ of Julia arcs that depend only on $\eta$ and have the following property.
%a $K$-quasiconformal extension to $\C$. 
For all $n\ge N$,
for any $J\in\Pi_n$, if $\tilde J=\xi(J)$, then there exists $M\in\N_0$ with  $\tilde\Lambda=f^M(\tilde J)\in\mathcal F$. Moreover, $\tilde\Lambda$ is the closure of a connected component of $\Ju\setminus\{\tilde z_1,\tilde z_2\}$ for some global cut-points $\tilde z_1, \tilde z_2\in\dee U_0$, and  
$f^{-M}$ is conformal in a neighborhood of $\tilde\Lambda$. 
\end{lemma}
\no
{\emph Proof.}
We first note that $\xi$ has a $K$-quasiconformal extension to $\C$, where $K$ depends only on $\eta$. Indeed, it follows from Lemma~\ref{L:EmbQS} that the homeomorphism $\xi_\infty$ of $\T$ defined by  
$$
\psi_\infty\circ\xi_\infty=\xi\circ\psi_\infty
$$
has a $K$-quasiconformal extension to $\C\setminus\overline{\D}$, where $K$ depends only on $\eta$. Since $\phi_\infty$ is conformal in $U_\infty$, it implies that $\xi$ has an extension to $U_\infty$ as a $K$-quasiconformal homeomorphism. Likewise, the discussion preceding Lemma~\ref{L:EmbQS} implies that $\xi$ has $K$-quasiconformal extension to each bounded Fatou component, perhaps with a different $K$ but that depends only on $\eta$. By \cite{Jo}, the Julia set $\Ju$ is removable for quasiconformal maps, and therefore the claim follows.   

Assume that the Julia arc $J$ in the statement of the lemma is determined by adjacent external rays $\rho_1, \rho_2\in R_n$ with landing points $z_1, z_2$, respectively.
There are two cases to consider: $z_1=z_2$ and $z_1\neq z_2$.

If $z=z_1=z_2$, then $z$ is the root of a bounded Fatou component $U$ such that $\dee U\subseteq J$, and $J=J_z=\dee\KK_z$, where $\KK_z$ is the limb rooted at $z$. 
Since any homeomorphism of $\Ju$ preserves the set of global cut-points, $\xi(z)$ is such a point. Moreover, from~\cite[Proposition~10.8]{He} applied to the quasisymmetry  $\xi$ it follows that there exists $N\in\N_0$ that depends only on $K$ such that if $n\ge N$, then $J\in\Pi_n$ implies that  the Julia arc $\tilde J$ does not contain $\dee U_0$. Roughly speaking, \cite[Proposition~10.8]{He} states that if $A$ and $B$ are overlapping sets in a metric space with the diameter of $A$ being smaller than the diameter of $B$, then an application of a quasisymmetric map with controlled distortion cannot spoil this relationship between the diameters of the corresponding images of $A$ and $B$ too much, quantitatively. 
In particular, it follows from the proof of  Lemma~\ref{L:Basepoint} that $\xi(z)$ is the root of the bounded Fatou component $\tilde U$ such that $\dee\tilde U=\xi(\dee U)$. Let $M\in\N_0$ be the dynamical distance from $\tilde U$ to $U_0$. 
Then $\tilde\Lambda=f^M(\tilde J)=\Ju_0$, and the claim follows in this case, with $\mathcal F$ consisting of the single element $\Ju_0$. We have $\tilde z_1=\tilde z_2=\alpha$. The fact that $f^{-M}$ is conformal in a neighborhood of $\Ju_0$ follows from the assumption that $M$ is the dynamical distance from $\tilde U$ to $U_0$.

Now assume that $z_1\neq z_2$. We know from Lemma~\ref{L:Adj} that there exists a bounded Fatou component $U$ such that  $z_1, z_2\in\dee U$. As above, there exists $N\in\N_0$ that depends only on $K$ such that if $n\ge N$, then for $J\in\Pi_n$ we have that $\tilde J$ does not contain $\dee U_0$.
Let $\Ju_{z_1}=\dee\KK_{z_1}$ and $\Ju_{z_2}=\KK_{z_2}$, where $\KK_{z_1}, \KK_{z_2}$ are the two {limbs} that are attached to $J$ at $z_1$ and $z_2$, respectively. 

The egg yolk principle
%~\cite[Theorem~11.14]{He} 
applied to an appropriate branch of $f^{-n}$  implies that ${\rm diam} \Ju_{z_i}\ge c\, {\rm diam} J,\ i=1,2$, for some $c>0$ that does not depend on $n$ or $J$. 
Let $\tilde\Ju_i=\xi(\Ju_{z_i}),\ i=1,2$.
Since $\xi$ is quasisymmetric,  an application of~\cite[Proposition~10.8]{He} gives that ${\rm diam}\tilde\Ju_i\ge \tilde c\, {\rm diam} \tilde J,\ i=1,2$, where the constant $\tilde c$ depends only on $K$. Since $f$ is hyperbolic, there exist constants $C=C(f)>0$ and  $M\in\N_0$ with the following properties. First, for $\tilde\Lambda=f^M(\tilde J)$ we have ${\rm diam}(\tilde\Lambda)\ge C$. Then,
there exist a point $p\in \tilde\Lambda$ and $r>0$,  such that  $\tilde\Lambda$ is contained in the disc $B(p,r/4)$. Finally, the disc $B(p,r)$ does not contain 0 and 1, i.e., the postcritical points of $f$. In particular, $r\ge 2C$. Since there are only finitely many large limbs of $\Ju$, the claim about the existence of a finite family $\mathcal F$ would follow if we show that there is a lower bound depending only on $K$ for ${\rm diam}f^M(\tilde\Ju_i),\ i=1,2$.  If both $f^M(\tilde\Ju_1)$ and $f^M(\tilde\Ju_2)$, are not contained in $B(p,r/2)$, then 
$$
{\rm diam}f^M(\tilde\Ju_i)\ge r/4\ge C/2,\ i=1,2.  
$$
If $f^M(\tilde\Ju_i)$ is contained in $B(p,r/2)$ for some $i=1,2$, then the claim follows from the egg yolk principle applied to $f^{-M}$ in the disc $B(p,r)$, and the fact that ${\rm diam}\tilde\Ju_i\ge \tilde c\, {\rm diam} \tilde J,\ i=1,2$.

The claim that, in the case $z_1\neq z_2$, the map $f^{-M}$ is conformal in a neighborhood of $\tilde\Lambda$ follows from the assumption that $\tilde\Lambda$ is contained in $B(p,r/4)$ and $B(p,r)$ does not contain either 0 or 1.

Finally, $\tilde\Lambda$ is the closure of a connected component of $\Ju\setminus\{\tilde z_1,\tilde z_2\}$, where $\tilde z_1=f^M(\xi(z_1)), \tilde z_2=f^M(\xi(z_2))$, and these points belong to the boundary of a bounded Fatou component $V$. The component $V$ does not have to be $U_0$. However, since ${\rm diam}(\tilde \Lambda)\geq C=C(f)$, the dynamical distance $m\in\N_0$ from $V$ to $U_0$ depends only on $C$. The set $f^m(\tilde\Lambda)$ is a connected component of $\Ju\setminus\{f^m(\tilde z_1),f^m(\tilde z_2)\}$, and  $f^m(\tilde z_1),f^m(\tilde z_2)$ belong to $\dee U_0$. Moreover, if we replace each element  $\tilde\Lambda\in\mathcal F$ by $f^m(\tilde\Lambda)$, all the other properties will be unchanged since $m$ depends only on $C$, and hence only on $f$. This completes the proof. 
\qed

\medskip

\begin{lemma}\label{L:LargeScale}
Let $\Lambda=\Ju_0$, and  let 
$\tilde\Lambda$ be a Julia arc from the finite family $\mathcal F$ guaranteed by Lemma~\ref{L:ConfElev}. 
%Let $\rho_1,\rho_2$ be two external rays in $U_\infty$ with landing points $z_1,z_2$, respectively, such that $\tilde\Lambda$ is the closure of one of the connected components of $\Ju\setminus\{z_1,z_2\}$.
Then there exists a  piecewise dynamical homeomorphism $\mathcal T$ of $\Lambda\setminus\{\alpha\}$ 
onto $\tilde\Lambda\setminus\{z_1,z_2\}$.
% such that $\mathcal T$ is piecewise dynamical.
%, i.e.,  $\Lambda$ can be subdivided into finitely many connected components by global cut-points $z_i\in\dee U_0, i=1,2,\dots k$, such that on each such component $\mathcal T$ agrees with a composition of integer iterates of $f$.  
\end{lemma}
\noindent
\emph{Proof.}
%From Lemma~\ref{L:Adj} and the assumption that $\xi$ from Lemma~\ref{L:ConfElev} is a homeomorphism, we conclude that $z_1,z_2$ belong to the boundary of the same bounded Fatou component $U$. By possibly applying an iterate of $f$ we may assume without loss of generality that $U=U_0$. 
If, in the notation of Lemma~\ref{L:ConfElev}, $\tilde z_1=\tilde z_2$, the present lemma is trivial. Otherwise,
the intersection of $\tilde\Lambda$ with $\dee U_0$ is an arc with dyadic endpoints. In the B\"ottcher coordinate of $U_0$ the intersection $\Lambda\cap\left(\dee U_0\setminus\{\alpha\}\right)$ corresponds to the open arc $I_0=\T\setminus\{1\}$, and the intersection $\tilde\Lambda\cap\left(\dee U_0\setminus\{\tilde z_1, \tilde z_2\}\right)$
corresponds to an open arc $I_1$ on $\T$. Thus we can apply Lemma~\ref{L:LargeScaleAppr} to get a piecewise linear map $\tau$  from $I_0$ onto $I_1$ with only dyadic break points  and with slopes that are integer powers of $2$.
The map $\tau$ corresponds to a piecewise dynamical map $\mathcal T$ from $\Lambda$ onto $\tilde\Lambda$ as described in Section~\ref{S:TG}, specifically in Lemma~\ref{lem:thom}.
\qed

\medskip

Let $\Lambda=\Ju_0$.  
For $n$ large as in Lemma~\ref{L:ConfElev}, 
we consider a tiling % subdivision 
of $\Ju$ into the  Julia arcs $J_i,\ i=1,2,\dots, 2^{n+1}$, of $\Pi_n$. 
For each $i=1,2,\dots, 2^{n+1}$, there exists $n_i\in \{ n, n+1\}$ such that 
$f^{n_i}$  % either  $f^n$ or $f^{n+1}$ 
maps $J_i$ onto $\Lambda$.
%  of the connected component of $\Ju\setminus\{\alpha\}$ that contains $\dee U_0\setminus\{\alpha\}$. 
% Without loss of generality we assume that it is $f^n$. 
% Suppose further that $J_i,\ i=1,2,\dots, 2^{n+1}$, is determined by the pair of adjacent external rays 
% $\rho_i^1, \rho_i^2$ with the landing points $z_i^1, z_i^2$, respectively. 
% Let $\tilde z_i^j =\xi(z_i^j),\ j=1,2$.
Let $\tl J_i = \xi(J_i) $, and  $\tilde\Lambda_i=f^{M_i} (\tl J_i)$, 
where $M_i$ is the constant from Lemma~\ref{L:ConfElev} corresponding to $\tl J_i $. 
%Also, for a Julia arc $J$, we let 
%  $\inter J$ be its {intrinsic} interior.
Now, let us replace the map $\xi$ on $J_i\setminus\{z_1^i, z_2^i\}$, where $z_1^i, z_2^i$ are the endpoints of $J_i$,
 by the map  $\tau_i=f^{-M_i}\circ \mathcal T_i\circ f^{n_i}$, where 
$\mathcal T_i$ is the piecewise dynamical map from $\Lambda\setminus\{\alpha\}$ 
onto $\inter \tl\Lambda_i$ that comes from Lemma~\ref{L:LargeScale}.   
%  $\tilde\Lambda_i\setminus\{f^{M_i} (\tilde z_i^1),f^{M_i} (\tilde z_i^2)\}$, 
Pasting these $2^{n+1}$ maps together,
%Doing this for each $i=1,2,\dots, 2^{n+1}$, 
% and denote the global map obtained in this way by $\tau$.  
we obtain a global piecewise dynamical map $\tau\in\Dyn(\Ju)$.

We next prove that $\tau$ is quasisymmetric on $\Ju$ with a  controlled distortion function $\eta$ by showing that it has a quasiconformal extension to the whole plane with a controlled dilatation $K$.  We do this by showing that $\tau$ can be quasiconformally extended into each bounded Fatou component and into the basin at infinity, so that the dilatation is controlled. 
Let $U$ be an arbitrary bounded Fatou component of $f$. If the dynamical distance $m$ from $U$ to $U_0$ is greater than $\lfloor (n+1)/2\rfloor$, the map $\tau$ has a conformal extension into $U$. 
Indeed, it is defined in $\dee U$ as a single element of the pseudo-group $\Ups_f$. 
%composition of integer iterates of $f$. 
Assume now that $m$ is at most $n$. Then the boundary of $U$ is partitioned by its intersection with the elements from $\Pi_n$ into the dyadic intervals of level ${\lfloor (n+1)/2\rfloor -m}$.  Let $t$ be the piecewise dynamical map of $\dee U$ that is the restriction of $\tau$ to $\dee U$.  It % necessarily 
agrees with $\xi$ at the endpoints of the dyadic intervals above.  
Now we need % the following 
two more lemmas concerning circle maps.

For a homeomorphism $\xi$ of $\T$ and a finite set $E\subset\T$ we denote by $\xi_E$ the linear interpolation of the restriction $\xi|_E$ of $\xi$ to $E$.

\begin{lemma}\label{L:Main}
For any distortion homeomorphism $\eta$,
 there exists a distortion homeomorphism $\eta'$ with the following property. If $\xi$  is an $\eta$-quasisymmetric map of the unit circle $\T$ and $E$ is a finite subset of $\T$ such that all the complementary intervals of $E$ in $\T$ have the same angular length, then $\xi_E$ is $\eta'$-quasisymmetric.
\end{lemma}
\noindent
{\emph Proof.}
Let $d$ denote the angular distance on $\T$.  Since $\xi$ is $\eta$-qua\-si\-sym\-met\-ric,  \cite[Proposition~10.8]{He} gives that there exists a constant $L\geq1$ that depends only on $\eta$ such that for any three distinct points $o, p, q\in E$ with $d(o,p)=d(o,q)$, we have 
\begin{equation}\label{E:qsineq}
d(\xi(o),\xi(p))\leq L d(\xi(o),\xi(q)).
\end{equation}

Let $a$ denote the angular length of each complementary interval of $E$. To verify that $\xi_E$ is quasisymmetric with a distortion function that depends only on $\eta$, it is enough to check that there exists a constant $C\ge 1$ that depends only on $\eta$ and has the following property.  If $o, p$, and $q$ are arbitrary points on $\T$ with $d(o,p)=d(o,q)=\delta$, for some $\delta>0$, then $d(\xi_E(o),\xi_E(p))\le C d(\xi_E(o),\xi_E(q))$.
We consider the following cases.

\noindent
\emph{Case 1:} $\delta\le a$ and $o, p, q$ are contained in the same complementary interval of $E$. This case is trivial and $C=1$ because $\xi_E$ is linear on each such interval.

\noindent
\emph{Case 2:} $\delta\le a$  and $o, p, q$ are not contained in the same complementary interval. We assume that $o$ and $p$ are contained in the closure of the same complementary interval $I$ of $E$ and $q$ is in the adjacent interval $I'$. Let us assume that $\xi_E$ scales $I$ by $s$ and it scales $I'$ by $s'$. We know from~\eqref{E:qsineq} that $s/L\le\ s'\le Ls$. Let $c\in E$ be the common point of the intervals $I$ and $I'$. Then 
$$
d(\xi_E(o),\xi_E(q))=d(\xi_E(o),\xi_E(c))+d(\xi_E(c),\xi_E(q))=sd(o,c)+s'd(c,q).
$$
For the last expression we have 
$$
sd(o,q)/L\le sd(o,c)+s'd(c,q)\le Lsd(o,q).
$$
Combining this with the assumptions that $d(o,p)=d(o,q)$ and that $\xi_E$ scales $I$ by $s$, we conclude that
$$
d(\xi_E(o),\xi_E(p))/L\le d(\xi_E(o),\xi_E(q))\le Ld(\xi_E(o),\xi_E(p)),
$$
and so $C=L$ in this case.

\noindent
\emph{Case 3:} $a<\delta\le3 a$. This case reduces to a repeated application of the proof of Case~2 at most 3 times. We  conclude that $C=L+L^2+L^{3}$ works in this case. 

\noindent
\emph{Case 4:} $\delta>3 a$. Let $I_o, I_p$, and $I_q$ be the closures of the complementary intervals of $E$ that contain $o, p$, and $q$, respectively. Since $\xi_E$ agrees with $\xi$ at the endpoints of each such interval, there exist $o', p'$, and $q'$ in $I_o, I_p$, and $I_q$, respectively, such that $\xi_E(o)=\xi(o'), \xi_E(p)=\xi(p')$, and $\xi_E(q)=\xi(q')$. Now,
$$
\begin{aligned}
\frac{d(\xi_E(o),\xi_E(p))}{d(\xi_E(o),\xi_E(q))}&=
\frac{d(\xi(o'),\xi(p'))}{d(\xi(o'),\xi(q'))}\leq \eta\left(\frac{d(o',p')}{d(o',q')}\right)\\
&\leq \eta\left(\frac{\delta+2a}{\delta-2a}\right)\leq\eta(5),
\end{aligned}
$$
and the claim follows in this case with $C=\eta(5)$.
\qed

\medskip

\begin{lemma}\label{L:BiLip}
Let $\mathcal L$ be a finite family of orientation preserving piecewise linear homeomorphisms between intervals of $\T$.
Let $\xi$ be an orientation preserving  homeomorphism of $\T$ and let $E$ be a finite subset of $\T$ such that all complementary intervals of $E$ in $\T$ have the same length. 
Suppose that $t$ is an orientation preserving  piecewise linear homeomorphism of $\T$ that agrees with $\xi$ on the set $E$. Moreover, assume that for each complementary interval $I$ of $E$  
 there are $M,n\in\N_0$ with $g^M\circ t\circ g^{-n}$ being an element of $\mathcal L$ defined on $g^n(I)$. 
Then there exists a constant $L\ge1$ that depends only on $\mathcal L$ such that the map $\xi_E^{-1}\circ t$ is $L$-bi-Lipschitz.
\end{lemma}
\no
{\emph{Proof.}}
The assumption implies, in particular, that there is a finite family $\mathcal L_0$ of orientation preserving  linear maps between intervals of $\T$, that depends only on $\mathcal L$, such that for each complementary interval $I$ of $E$ in $\T$ the map $g^M\circ\xi_E\circ g^{-n}$ is an element of $\mathcal L_0$. Therefore the map $g^n\circ\xi_E^{-1}\circ t\circ g^{-n}$ is a homeomorphism of $g^n(I)$ that belongs to a finite family, depending only on $\mathcal L$, of orientation preserving piecewise linear homeomorphisms. Hence there exists $L\ge1$ that depends only on $\mathcal L$ such that $g^n\circ\xi_E^{-1}\circ t\circ g^{-n}$ is an $L$-bi-Lipschitz homeomorphism of $g^n(I)$. Since $g$ is the scaling map (by the factor 2) in the angular metric, the map $\xi_E^{-1}\circ t$ is $L$-bi-Lipschitz on each $I$, and therefore on all of $\T$.
\qed

\medskip

Lemmas~\ref{L:Main} and~\ref{L:BiLip} imply that the restriction of the map $\tau$ to the boundary of each bounded Fatou component $U$ has a $K'$-quasiconformal extension to $U$, where $K'$ depends only on $\eta$. The set $E$ in Lemma~\ref{L:Main} is the set of dyadic points at level ${\lfloor (n+1)/2\rfloor -m}$,
where $m$ is the dynamical distance from $U$ to $U_0$.  The finite family $\mathcal L$ in Lemma~\ref{L:BiLip} comes from Lemmas~\ref{L:ConfElev} and~\ref{L:LargeScale}. Indeed, Lemma~\ref{L:ConfElev} guarantees the existence of a finite family $\mathcal F$ of subsets of $\Ju$ such that for any Julia arc $J\in\Pi_n$ we have $\tilde\Lambda=f^M(\xi(J))\in\mathcal F$ for some $M\in\N_0$. Lemma~\ref{L:LargeScale} then gives a piecewise dynamical map $\mathcal T$ from $\Lambda\setminus\{\alpha\}$ onto $\tilde\Lambda\setminus\{z_1,z_2\}$, in the notations of that lemma. Because the family $\mathcal F$ is finite, the family of such maps $\mathcal T$ is finite.
Now, if the defining external rays of $J$ land on the  boundary of $U$ and if we assume for simplicity that $\Lambda=f^n(J)$ rather than $\Lambda=f^{n+1}(J)$, then  we have
$
\mathcal T=f^M\circ\tau\circ f^{-n},
$
where the inverse branches of $f$ are chosen appropriately. Passing to the B\"ottcher coordinates of $U$ and $\tilde U=\xi(U)$ we conclude that the map
$
g^{M-\tilde m}\circ t\circ g^{-(n-m)},
$
belongs to a finite family $\mathcal L$, where $t$ is the conjugate map of the map $\tau$ by $\phi_U$, and $\tilde m$ is the dynamical distance from $\tilde U$ to $U_0$. Since $\xi_E$ is $\eta'$-quasisymmetric and $\xi_E^{-1}\circ t$ is $L$-bi-Lipschitz, we conclude that $t=\xi_E(\xi_E^{-1}\circ t)$ is $\eta'\circ L^2$-quasisymmetric, where $L^2$ is the scaling map by $L^2$.  
Thus, the Ahlfors--Beurling extension~\cite{BA} gives that there exists $K'\ge1$ that depends only on $\eta'$ and $L$, and hence only on $\eta$, such that the map $t$ has a $K'$-quasiconformal extension into $\D$. Conjugating back via B\"ottcher coordinates we conclude that $\tau$ has a $K'$-quasiconformal extension into every bounded Fatou component $U$. 

To deal with the unbounded component $U_\infty$, we first apply Lem\-ma~\ref{L:OuterDyadSubdiv} 
and then proceed in the same way as for bounded Fatou components, 
i.e., using  Lemmas~\ref{L:Main} and~\ref{L:BiLip}. 
The finite set $E$ in this case is the set $\psi(D_n)$, where $\psi$ is the bi-Lipschitz map guaranteed by Lemma~\ref{L:OuterDyadSubdiv}. 
The crucial difference is that, unlike the case of bounded Fatou components, 
the boundary of $U_\infty$ touches itself and one has to be careful to preserve the basilica lamination. 
By Lemma~\ref{homeos}, 
any topologically extendable  homeomorphism $\xi$ of $\Ju$ induces a homeomorphism $\xi_\infty$ of $\T$ 
that preserves the basilica lamination. 
% lamination. 
% This follows from the fact that such a $\xi$ preserves the cyclic order of limbs of $\Ju$ at each global cutpoint, 
% combined with the prime ends topology.
By Lemma \ref{L:EmbQS}, $\xi_\infty$  is  $\eta'$-quasisymmetric  (with $\eta'$ depending only on $\eta$).
As $\xi$ is replaced by a piecewise dynamical map $\tau$, the map $\xi_\infty$ is replaced by a Thompson-like element $\tau_\infty$ satisfying 
$
\psi_\infty\circ\tau_\infty=\tau\circ\psi_\infty.
$
Such a map $\tau_\infty$ necessarily respects the basilica lamination.
By  Lemmas~\ref{L:Main} and~\ref{L:BiLip}, the map $\tau_\infty$ is $\eta''$-quasisymmetric with $\eta''$ that depends only on $\eta'$ and hence only on $\eta$. Therefore it  extends to a $K''$-quasiconformal homeomorphism of $\C\sm \D$
(with $K''$ depending only on $\eta''$, and ultimately, only on $\eta$).
%Applying the second part of Lemma \ref{homeos}, we conclude that $\xi$ extends to a $K$-quasiconformal homeomorphism of $U_\infty$.   
% Conversely, if $\xi_\infty$ is any homeomorphism of $\T$ that preserves the basilica lamination and extends to an orientation preserving home%omorphism of $\C\setminus\D$, this extension conjugates by $\phi_\infty$ to a homeomorphism of $\overline{U_\infty}$. This comes from the fact% that the Julia set $\Ju$ is the factor space of $\T$ with respect to the basilica lamination and we use the factor space topology. 

Thus, the map $\tau$ has a $K''$-quasiconformal extension to $U_\infty$ as well, 
where $K''$ depends only on $\eta$. Putting this together, we obtain a $K$-quasiconformal extension of $\tau$ from $\Ju$ into each Fatou component of $f$, where $K=\max\{K',K''\}$. We  denote this extension by $\tau$ as well.
As above, the polynomial $f$ is postcritically finite, and, according to~\cite{Jo}, the Julia set $\Ju$ is removable for quasiconformal maps. The map $\tau$ is hence $K$-quasiconformal in the whole complex plane, and therefore $\eta'$-quasisymmetric for some $\eta'$ that depends only on $\eta$. Its restriction $\tau$ to $\Ju$ is thus also $\eta'$-quasisymmetric.

\begin{lemma}\label{lem:gps}
Any piecewise dynamical map $\tau$ of $\Ju$ belongs to the group $\hat T$ generated by $T$ and $\iota$.
\end{lemma}
\no
{\emph{Proof.}}
It is proved in Lemma~\ref{L:Trans}  that the group $\hat T$ 
 acts transitively on the set of  bounded Fatou components. 
Suppose that $\tau(\dee U_0)=\dee \tilde U_0$ and let $g\in \hat T$ be such that $g(\dee U_0)=\dee \tilde U_0$. Then $g^{-1}\circ\tau$ keeps $\dee U_0$ invariant, and so, without loss of generality, we assume that $\tau$ keeps $\dee U_0$ invariant. Note that Lemma~\ref{L:Basepoint} implies that such $\tau$ sends the root of each $U\neq U_0$ to the root of $\tilde U=\tau(U)$.

Let $U$ be any bounded Fatou component of $f$ 
and $\tilde U=\tau(U)$. 
Let $\phi_U$ and $\phi_{\tilde U}$ be the B\"ottcher coordinates of $U$ and $\tilde U$, respectively.  
Since $\tau$ is assumed to be piecewise dynamical, the restriction of the map $\phi_{\tilde U}\circ\tau\circ\psi_U$ to the unit circle is an element of the Thompson group $T$ of the unit circle. Here $\psi_U=\phi_U^{-1}$. Moreover,  this map respects the B\"ottcher coordinate for all but finitely many bounded Fatou components $U$. This follows immediately from the fact that there are only finitely many global cut-points used in the definition of a piecewise dynamical map.
Let $t$ be the smallest subtree of $\TT$ that contains $U_0$ and such that $\tau$ respects the B\"ottcher coordinate of every bounded Fatou component $U$ that is not a vertex of $t$. 
Let $k$ be the number of vertices of $t$. 
We prove by induction on $k$ that there is an element $g$  of $\hat T$ such that  $\tau=g\in \hat T$. If $k=1$, then there exists $\theta\in T$ such that $\theta^{-1}\circ \tau$ is the identity element on $\dee U_0$. Moreover, since $\theta$ respects the B\"ottcher coordinate of each bounded Fatou component $U\neq U_0$, then $\theta^{-1}\circ \tau$  is necessarily the identity in the B\"ottcher coordinate of every $U$, and so it is the identity.

Now suppose the result is true for $k-1$. Let $U$ be a leaf of $t$, i.e., a degree one vertex, and let $\tilde U$ be the leaf of $\tilde{t}$ that corresponds to $U$ under $\tau$. Let $p$ be the unique path from $U_0$ to $U$ in $t$. By Lemma~\ref{L:Trans}, there is an element $g_U\in \hat T$ such that $g_U(\dee U)=\dee U_0$ and $g_U$ respects the  B\"ottcher coordinates of all bounded Fatou components $V$ with the possible exception of $V$ being a vertex of $p$. Likewise, there is an element $g_{\tilde U}\in \hat T$ that has the same properties with respect to $\tilde U$ and $\tilde{t}$. Also, $g_U$ and $g_{\tilde U}$ respect the B\"ottcher coordinates of $U$ and $\tilde U$, respectively. 

Then the restriction of  $g_{\tilde U}\circ \tau\circ g_{U}^{-1}$ to $\dee U_0$ 
equals to the restriction of an element $\theta_U\in T$, and hence  $ g_{\tilde U}\circ \tau\circ g_{U}^{-1}\circ\theta_U^{-1}$ is the identity on $\dee U_0$. Moreover, since $\tau$ takes the root of $U$ to the root of $\tilde U$, the map $\theta_U$ fixes the root $\alpha$ of $U_0$. The map $g=g_{U}^{-1}\circ\theta_U\circ g_{U}$ is then an element of $\hat T$ that keeps $U$ invariant, fixes the root of $U$, 
and respects the B\"ottcher coordinate of every $V\neq U$. 
In the B\"ottcher coordinates of $U$ and $\tilde U$, the map $\tau\circ g^{-1}$ is the identity. Also, it is the identity in B\"ottcher coordinates of every vertex $V$ that is not in $t$. Now we can apply the induction.
\qed

\medskip

Recall that in~\cite{BF} the authors studied the Thompson-like group $T_\bal^b$ 
of piecewise linear homeomorphisms of the unit circle that preserve the basilica lamination. %It is immediate that 
Obviously, our group $\hat T$ generated by $T$ and $\iota$ is isomorphic to a subgroup of $T_\bal^b$. 
The following is an immediate corollary of Lemmas~\ref{L:homeos2} and~\ref{lem:gps}. 
%the fact that the Julia set $\Ju$ is the boundary of $U_\infty$.
\begin{corollary}\label{C:gps}
The groups $\hat T$ and $T_\bal^b$ are isomorphic.
\end{corollary}  

\medskip

To finish the proof of Theorem~\ref{thm:thom} we need to show that $\tau$ approaches $\xi$ on $\Ju$ uniformly as $n$ goes to infinity.
Indeed, as one moves further away from the vertex of $\TT$ that corresponds to $U_0$, the diameters of the corresponding bounded Fatou components go to 0. This follows from the fact that $f$ is hyperbolic. Moreover, the diameter of each $J_k\in\Pi_n$ goes to 0 as $n\to\infty$.
This, along with the fact that $\tau$ agrees with $\xi$ at all the preimages of $\alpha$ under $f^n$, finishes the proof of Theorem~\ref{thm:thom}.
\qed

\end{document}